\documentclass[12pt,oneside,reqno]{amsart}
\usepackage{appendix}
      \usepackage{amssymb}
\usepackage{color}
\usepackage{dsfont}
\usepackage[pdftex]{graphicx}
      \theoremstyle{plain}
      \newtheorem{theorem}{Theorem}[section]
      \newtheorem{lemma}[theorem]{Lemma}
      \newtheorem{proposition}[theorem]{Proposition}

      \theoremstyle{definition}

      \theoremstyle{remark}

\begin{document}

%


\title[Asymptotic direction for RWRE.]
{Asymptotic direction for  random walks in  mixing random environments}


   \author{Enrique Guerra$^{1}$ and Alejandro F. Ram\'\i rez$^{1,2}$
  }

\thanks{
$^1$ Partially supported by Iniciativa Cient\'\i fica Milenio NC120062}

\thanks{
$^2$ Partially supported  by Fondo Nacional de Desarrollo Cient\'\i fico
y Tecnol\'ogico grant 1141094.}

\email{eaguerra@mat.uc.cl, aramirez@mat.puc.cl}
\address{ Facultad de Matem\'aticas\\
Pontificia Universidad Cat\'olica de Chile\\
Vicu\~na Mackenna 4860, Santiago 7820436, Chile\\
Telephone: [56](2)2354-5466\\
Telefax: [56](2)2552-5916}




   \begin{abstract}
We prove that every random walk in a uniformly elliptic random environment satisfying
the cone mixing condition and a non-effective polynomial
ballisticity condition with high enough degree has an asymptotic direction.
    \end{abstract}

   \date{\today}


   \maketitle

\noindent {\footnotesize

}


   \section{Introduction}
Random walk in random environment is a
 simple but powerful model for a variety of
phenomena including
homogenization in disordered materials \cite{M94}, DNA chain replication \cite{Ch62},
crystal growth \cite{T69} and turbulent behavior in fluids
\cite{Si82}. 
 Nevertheless, challenging and fundamental questions about
it remain open (see \cite{Z04} for a
general over\-view). 
In the multidimensional setting 
 a widely open question is to  
establish relations between the environment at
a local level and the long time behavior of the random walk.
During
last ten years, interesting progress has been achieved
specially in the case in which the movement takes
place on the hypercubic lattice $\mathbb Z^d$ and the environment is  i.i.d.,
establishing relations between directional transience, ballisticity
and the existence of an asymptotic direction and the law
of the environment in finite regions.
To a great extent, these arguments are no 
longer valid when the i.i.d. assumption is dropped.

In this article we focus on the problem of finding local
conditions on the environment which ensure the
existence of a deterministic asymptotic direction for the random walk model
in contexts where the environment  
 is not necessarily i.i.d. As it will be shown in Section \ref{sex},
there exist environments which are ergodic and for which there
does not exist a deterministic asymptotic direction. Therefore,
some kind of mixing or ballisticity condition should be imposed on
the environment.
Here we establish the existence of an asymptotic
direction for random walks in random environments
which are uniformly elliptic, are  cone mixing \cite{CZ01}, and satisfy
a non-effective version of the polynomial ballisticity condition
introduced in \cite{BDR14} with high enough degree of the decay.
It will be also shown (see Section \ref{sex}), that there
exist environments {\it almost} satisfying the   above assumptions which
are directionally transient and for which there exists
at least in a weak sense an asymptotic direction, but have a vanishing velocity. 
Here the term {\it almost} is used
because in these examples the non-effective polynomial ballisticity condition is satisfied
 with a low degree. This
shows that somehow, while   
the  mixing and non-effective polynomial conditions we
will impose do imply the existence of an asymptotic direction,
they might not necessarily imply the existence of a non-vanishing
velocity.

For $x\in\mathbb R^d$, we denote by $|x|_1$, $|x|_2$ and $|x|_\infty$ its $l_1$,  $l_2$ and $l_\infty$ norms respectively.
For each integer $d\ge 1$, we consider the $2d-$dimensional simplex
$\mathcal{P}_{d}:=\{z\in (\mathbb{R}^+)^{2d}: \sum_{i=1}^{2d} \  z_i=1 \}$
and $E:=\{e\in\mathbb Z^d: |e|_1=1\}$.
We define the {\it environmental space}
$\Omega:= \mathcal{P}_{d}^{\mathbb{Z}^d}$ and endow it with
its product $\sigma$-algebra. Now, for a fixed
 $\omega=\{\omega(y):y\in\mathbb Z^d\}\in \Omega$, with $\omega(y)=\{\omega(y,e):e\in U\}\in\mathcal P_d$, and a fixed $x\in\mathbb Z^d$, we consider the Markov
chain $\{X_n:n\ge 0\}$ with state space $\mathbb{Z}^d$
starting from $x$ defined by the transition probabilities

\begin{equation}
\label{qlaw}
P_{x,\omega}[X_{n+1}=X_n + e \mid X_n]=\omega(X_n,e)\qquad {\rm for}\quad   e\in U.
\end{equation}
We denote by $P_{x,\omega}$ the law of this Markov chain and call it
a random walk in the environment $\omega$.
 Consider a law $\mathbb P$
defined on $\Omega$. We call $P_{x,\omega}$ the {\it quenched law}
of the random walk starting from $x$. Furthermore, we define
the semi-direct product probability measure on $\Omega\times(\mathbb Z^d)^{\mathbb N}$ by

$$
P_x(A\times B):=\int_A P_{x,\omega}(B)d\mathbb P
$$
for each Borel-measurable set $A$ in $\Omega$ and $B$ in $(\mathbb Z^d)^{\mathbb N}$,
 and call it the {\it annealed or averaged law}
of the random walk in random environment. The law $\mathbb P$
of the environment is said to be i.i.d. if the random variables
$\{\omega(X):x\in\mathbb Z^d\}$ are i.i.d. under $\mathbb P$,
 {\it elliptic} if for every $x\in\mathbb Z^d$ and $e\in U$
one has that $\mathbb P[\omega(x,e)>0]=1$ while uniformly
elliptic if there exists a $\kappa>0$
such that
 $\mathbb P[\omega(x,e)\ge\kappa]=1$ for every $x\in\mathbb Z^d$ and $e\in U$.

Let $l\in\mathbb S^{d-1}$.
We say that a random walk is {\it transient in direction $l$}
or just {\it directionally transient} if $P_0$-a.s. one has that

$$
\lim_{n\to\infty}X_n\cdot l=\infty.
$$
Furthermore, we say that it is {\it ballistic in direction $l$} if

$$
\liminf_{n\to\infty}\frac{X_n\cdot l}{n}>0.
$$
In the case in which the environment is elliptic and i.i.d., it is known
that whenever a random walk is ballistic necessarily a law of large
numbers is satisfied and in fact $\lim_{n\to\infty}\frac{X_n}{n}=v\ne 0$
is deterministic \cite{DR14}. Furthermore, in the uniformly
elliptic i.i.d. case, it is still an open question to establish
whether or not in dimensions $d\ge 2$, every directionally transient
random walk is ballistic (see \cite{BDR14}).

On the other hand, we say that $\hat v\in\mathbb S^{d-1}$ is
an {\it asymptotic direction} if $P_0$-a.s. one has that

$$
\lim_{n\to\infty}\frac{X_n}{|X_n|_2}=\hat v.
$$
For elliptic i.i.d. environments,
 Simenhaus  established \cite{Si07} the
existence of an asymptotic direction
whenever the random walk is directionally transient in an open set
of ${\mathbb S^{d-1}}$. As it will be shown in Section \ref{sex}, this
statement is not true anymore when the environment is assumed to
be ergodic instead of i.i.d., even if it is uniformly elliptic.

 Let us now define the three main assumptions throughout this
article: uniform ellipticity, cone mixing, and non-effective
polynomial ballisticity condition.
 Let $\kappa>0$. We say that
$\mathbb{P}$  is {\it uniformly elliptic with respect to $l$}, denoted
by $(UE)|l$, if
 the jump
probabilities of the random walk are positive and larger  than $2\kappa$
 in those directions which for which the projection of $l$ is positive.
In other words if
$\mathbb{P}[\omega(0,e)>0]=1$ for $e\in E$
and if

$$
\mathbb{P}\left[\min_{e\in\mathcal{E}}\ \omega(0,e)\geq 2\kappa\right]=1,
$$
where

\begin{equation}
\label{espacio-epsilon}
\mathcal E:=\cup_{i=1}^d\{sgn(l_i)e_i\}-\{0\}
\end{equation}
and by convention $sgn(0)=0$.

\medskip

We will now introduce a certain mixing assumption for the
environment $\mathbb P$.
Let  $\alpha>0$ and $R$ be a rotation such that

\begin{equation}
\label{rot}
R(e_1)=l.
\end{equation}
To define the cone, it will be useful to consider for each $i \in [2,d]$,

$$
l_{+i}=\frac{l+\alpha R(e_i)}{|l+\alpha R(e_i)|}\qquad
{\rm and}\qquad
l_{-i}=\frac{l-\alpha R(e_i)}{|l-\alpha R(e_i)|}.
$$
The cone $C(x, l,\alpha)$ centered in $x\in \mathbb R^d$ is defined as

\begin{equation}
\label{cono}
C(x, l,\alpha)):=\bigcap_{i=2}^{d}\left\{z \in \mathbb{R}^d: (z-x) \cdot l_{+i} \geq 0, \, (z-x) \cdot l_{-i} \geq 0\right\}.
\end{equation}
Let $\phi: [0,\infty)\to [0,\infty)$
be such that $\lim_{r\to\infty}\phi(r)=0$.
We say that a stationary probability measure $\mathbb{P}$  satisfies the
{\it cone mixing assumption} with
respect to $\alpha$, $l$ and $\phi$, denoted $(CM)_{\alpha,\phi}|l$, if for every pair of events $A, B$, where $\mathbb{P}(A)>0$, $A\in\sigma\{\omega(z,\cdot); z\cdot l\leq 0\}$, and $B\in \sigma\{\omega(z,\cdot); z\in C(rl,l,\alpha)\}$, it holds that
$$
\left|\frac{\mathbb{P}[A\cap B]}{\mathbb{P}[A]}-\mathbb{P}[B]\right|\leq \phi(r |l|_1).
$$
We will see that every stationary cone mixing measure $\mathbb P$ is
necessarily ergodic. On the other hand, a cone-mixing
environment can be such that the jump probabilities are highly dependent along
certain directions.

\medskip

We now introduce an assumption which is closely related
to the effective polynomial ballisticity condition introduced in
\cite{BDR14}.
For each $A\subset \mathbb{Z}^d$  we define

$$
\partial A:=\{z\in\mathbb{Z}^d: z\not\in A, \mbox{ there exists some  }y\in A \mbox{ such that } |y-z|=1 \}.
$$
Define also the stopping time

$$
T_A:=\inf\{n\geq0: X_n\not\in A\}.
$$
Given $L,L'>0$, $x\in\mathbb Z^d$ and $l\in\mathbb S^{d-1}$ we define the boxes

$$
B_{L,L',l}(x):=
x+R\left(\left(-L,L\right)\times\left(-L',L'\right)^{d-1}\right)
\cap\mathbb{Z}^d,
$$
where $R$ is defined in (\ref{rot}).
The \emph{positive boundary} of $B_{L,L',l}(x)$, denoted by $\partial^+B_{L,L',l}(0)$, is

$$
\partial^+B_{L, L',l}(0):=\partial B_{L, L',l}(0)\cap\{z:z\cdot l\geq L\},
$$
Define also the half-space

$$
 H_{x,l}:=\{y\in\mathbb Z^d: y\cdot l< x\cdot l\},
$$
and the corresponding $\sigma$-algebra of the environment on that half-space

$$
\mathcal H_{x,l}:=\sigma(\omega(y):y\in H_{x,l}).
$$
Now, for $M\ge 1$, we say that the
{\it  non-effective polynomial} condition  $(PC)_{M,c}|l$ is satisfied if there exists some $c>0$ so that for  $y \in H_{0,l}$ one has that

\begin{equation}
\label{Pl}
\overline\displaystyle{\lim_{L\to\infty}}L^M \sup P_0\left[X_{T_{B_{L, cL,l}}(0)}\not\in \partial^+B_{L,
 cL,l}(0),
T_{B_{L, cL,l}(0)}<T_{H_{y,l}}|\mathcal H_{y,l}\right]= 0,
\end{equation}
where the supremum is taken  over all the coordinates $\{\omega(x): x\cdot l\le y\cdot l\}$.
 It is possible to show that for i.i.d. environments,
this condition is implied by Sznitman's $(T')$ condition \cite{Sz03}, and
it is equivalent to the effective polynomial condition introduced
in \cite{BDR14}.

Let $\mathbb I$ be the subset of vectors in $\mathbb R^d$ different from $0$ and with integer
coordinates.
Define $\mathbb S^{d-1}_q:=\left\{\frac{l}{|l|_2}:l\in\mathbb I\right\}$. We can now state  our main result.

\medskip

\begin{theorem}
\label{mainth}
Let $l\in\mathbb S^{d-1}_q$, $M> 6d$, $c>0$ and $0<\alpha\leq\min\{\frac{1}{9},\ \frac{1}{2c+1}\}$.
Consider a random walk in a random environment with stationary law
satisfying the
 uniform ellipticity condition $(UE)|l$, the cone mixing condition  $(CM)_{\alpha,\phi}|l$
and the  non-effective polynomial condition $(PC)_{M,c}|l$.
Then, there exists a deterministic $\hat v\in\mathbb S^{d-1}$ such that $P_0$-a.s. one has that

$$
\lim_{n\to\infty}\frac{X_n}{|X_n|}=\hat v.
$$
\end{theorem}

\medskip

\noindent  As it will be explained in Section \ref{sex},  Simenhaus's theorem which states that an asymptotic
direction exists whenever the random walk is directionally transient in an open set
of directions and the environment is i.i.d., is not true if the i.i.d. assumption
is dropped. Somehow, Theorem \ref{mainth} shows that if the i.i.d. assumption
is weakened to cone mixing, while directional transience is strengthened to
the non-effective polynomial condition, we still can guarantee the existence of
an asymptotic direction. 

In \cite{CZ01}, the existence of a strong law of large numbers is established for
random walks in cone-mixing environments which also satisfy a version of Kalikow's
condition, but under an additional assumption of existence of certain moments
of approximate regeneration times. This assumption is unsatisfactory in the sense
that it is in general difficult to verify if for a given random environment  it is true or not.
On the other hand, as it will be shown in Section \ref{sex}, there exist examples of random walks
in a random environment satisfying the cone-mixing assumption 
for which the law of large numbers is not satisfied, while an asymptotic direction
exists.
From this point of view, Theorem 1.1 is also a first step in the direction of
obtaining scaling limit theorems for random walks in cone-mixing environments
through ballisticity conditions weaker than Kalikow's condition, and without
any kind of assumption on the moments of approximate regeneration times or 
of the position of the random walk at these times. On the other hand,
in \cite{RA03}, a strong law of large numbers is proved for random walks
which satisfy Kalikow's condition and Dobrushin-Shlosman's strong mixing assumption.
The Dobrushin-Shlosman strong mixing assumption is stronger than cone-mixing, both because
it implies cone-mixing in every direction and because it corresponds to a decay
of correlations which is exponential.

A key step to prove Theorem 1.1 will be to establish that the
probability that the random walk never exits a cone is positive
through the use of renormalization type ideas, and only assuming
the non-effective polynomial condition and uniform ellipticity.
Using this fact, we will define approximate regeneration times
as in \cite{CZ01}, showing that they have finite moments
of order larger than one
when we also assume cone-mixing. This part of the proof will
require careful and tedious computations. Once this is done,
the law of large numbers can be deduced using for example the
coupling approach of \cite{CZ01}.

In Section \ref{sex}, we will present two examples of random walks
in random environments which exhibit a behavior which is
not observed in the i.i.d. case, giving an idea of the kind of limitations
given by the  framework
 of Theorem \ref{mainth}. In Section \ref{pddis}, the
meaning of the non-effective polynomial condition and its
relation to other ballisticity conditions will be discussed.
In Section \ref{pdd}, we will
show that the non-effective polynomial condition implies
that the probability that the random walk never exits a cone is positive.
This will be used in Section \ref{section5} to  prove that the
approximate regeneration times have finite moments of order larger than one.
Finally in Section \ref{finals}, Theorem \ref{mainth} will be proved
using coupling with i.i.d. random variables.

\medskip

\section{Examples of  directionally transient random walks
without  an asymptotic direction and vanishing velocity}
\label{sex}

We will present two examples of random walks in random environment
which exhibit the framework of the hypothesis of Theorem \ref{mainth}.
The first example  indicates that the hypothesis
of Theorem \ref{mainth} might not necessarily imply a
strong law of large numbers with a non-vanishing velocity.
The second example will show  that we cannot expect to prove the existence of an asymptotic
direction without either some kind of mixing hypothesis on the environment
or some ballisticity condition.

Throughout,
 $p$ will be a random variable taking values in $(0,1)$ such that
there exists a unique $\kappa\in (1/2,1)$ with the property that

\begin{equation}
\label{tnb}
E[\rho^\kappa]=1\quad {\rm and}\quad E[\rho^\kappa\ln^+\rho]<\infty,
\end{equation}
where $\rho:=(1-p)/p$.

\medskip
\subsection{Random walk with a vanishing velocity but with an asymptotic direction}

 Let $\{p_i:i\in\mathbb Z\}$ be i.i.d. copies
of $p$.
 Let $e_1$ and $e_2$ be the canonical vectors in $\mathbb Z^2$. Define an i.i.d. sequence of random
variables $\{\omega_i:i\in\mathbb Z\}$ with $\omega_i=\{\omega_i(e_1),
\omega_i(-e_1),\omega_i(e_2), \omega_i(-e_2)\}$,
by

$$
\omega_i(e_2)=\omega_i(-e_2)=\frac{1}{4},
$$

$$
\omega_i(e_1)=\frac{p_i}{2}\quad {\rm and}\quad
\omega_i(-e_1)=\frac{1}{2}-\frac{p_i}{2}.
$$
Now consider the  random environment
$\omega=\{\omega((i,j)):(i,j)\in\mathbb Z^2\}$ defined

$$
\omega((i,j)):=\omega_i\quad {\rm for}\ {\rm all}\quad i,j\in\mathbb Z.
$$
We will call $\mathbb P_1$ the law of the above environment and $Q_1$ the  annealed law of the corresponding random
walk starting from $0$.

\medskip

\begin{theorem}
\label{te1}
Consider a random walk in a random environment
with law $\mathbb P_1$. Then, the following are satisfied:

\begin{itemize}

\item[$(i)$] $Q_1$-a.s.

$$
\lim_{n\to\infty}X_n\cdot e_1=\infty.
$$
\item[$(ii)$] $Q_1$-a.s.

$$
\lim_{n\to\infty}\frac{X_n}{n}=0.
$$

\item[$(iii)$] In $Q_1$-probability

$$
\lim_{n\to\infty}\frac{X_n}{|X_n|_2}=e_1.
$$

\item[$(iv)$] The law  $Q_1$ satisfies the polynomial condition $(PC)_{M,c}$ with $M=\kappa-\frac{1}{2}-\varepsilon$ and $c=1$, where $\varepsilon$ is an arbitrary number in the interval $(0,\kappa-\frac{1}{2})$.

\end{itemize}
\end{theorem}
\begin{proof}
{\it Part (i).}
 We will describe a \emph{one dimensional procedure} which will be used throughout the proofs of items $(i)$ and $(ii)$. Define $\{Y_n:n\geq 0\}:=\{X_n\cdot 
e_1:n\geq0\}$. Note that
    \begin{eqnarray*}
    &P_{0,\omega}[Y_{n+1}=Y_n+e_1\mid Y_n]= \widetilde{\omega}(Y_n,e_1)=p_{Y_n}/2,\\
    &P_{0,\omega}[Y_{n+1}=Y_n-e_1\mid Y_n]=\widetilde{\omega}(Y_n,-e_1)=(1-p_{Y_n})/2 \mbox{, and }\\
    &P_{0,\omega}[Y_{n+1}=Y_n\mid Y_n]=\widetilde{\omega}(Y_n,0)=1/2.\\
    \end{eqnarray*}
    By (\ref{tnb}) it follows that $\widetilde{E}_1[\ln[\widetilde{\rho_0}]]<0$, where $\widetilde{\rho}_0:=\widetilde{\omega}(0,-e_1)/\widetilde{\omega}(0,e_1)$ and $\widetilde{E}_1$ denotes the corresponding expectation in this random environment. Now, from the transience criteria in \cite{Z04} Theorem 2.1.2 one has that $Q_1$- a.s.
    $$
    \lim_{n\to\infty}X_n\cdot e_1=\infty.
    $$

\smallskip

\noindent {\it Part (ii).} Note that

    $$
    \frac{X_n}{n}=\frac{Y_n e_1+ (X_n\cdot e_2)e_2}{n},
    $$
where $\{Y_n:n\ge 0\}$ is the projection of the random walk in the
direction $e_1$ defined in part $(i)$. 
Now, using the strong law of large numbers for this projection (\cite{Z04}, Theorem 2.1.9), we get $Q_1$-a.s., and the fact that $(X_n\cdot e_2)$ is
a random walk which moves with the same probability in both directions, we
conclude that $Q_1$-a.s.

$$
\lim_{n\to\infty}\frac{X_n}{n}=0.
$$

\smallskip
\noindent {\it Part (iii).} We define the random variables $N_1$ and $N_2$ as horizontal and vertical steps performed by the walk $X_n$, respectively. By the very definition of this example, both of them distribute like a binomial law of paratemers $n$ and $1/2$ under the quenched law.
For each $\varepsilon>0$, we have to estimate the  probability

    \begin{equation}
\label{qiu1}
   Q_1\left[\left| \frac{X_n}{|X_n|_2}
-e_1\right|>\varepsilon\right]
=Q_1\left[\left|\frac{\frac{(X_n\cdot e_1)}{n^{\kappa}}e_1+\frac{(X_n\cdot e_2)}{n^{\kappa}}e_2}{\sqrt{\frac{(X_n\cdot e_1)^2}{n^{2\kappa}} +\frac{(X_n\cdot e_2)^2}{n^{2\kappa}}}}-e_1\right|>\varepsilon\right].
    \end{equation}
    Clearly, $X_n \cdot e_2$ under the annealed law has the same law $\widetilde{P}$ of a one dimensional simple symmetric random walk $\{Z_{m}:m\ge 0\}$ at time $m=N_2$.
Note that $\widetilde{P}$- a.s. $N_2/n\rightarrow 1/2$ as $n\to\infty$. Therefore, since $\kappa>1/2$ we see that
    $$
    Q_1\left[\lim_{n\to\infty}\frac{X_n\cdot e_2}{n^{\kappa}}=0\right]=\widetilde{P}\left[\lim_{n\to \infty}\frac{Z_{N_2}}{N_2^{\kappa}}\frac{1}{2^\kappa}=0\right]=1.
    $$
    and hence $Q_1$-a.s.
    $$
    \lim_{n\to\infty}\frac{(X_n\cdot e_2)^2}{n^{2\kappa}}=0.
    $$
    On the other hand, using the convergence theorem of Kesten, Kozlov and Spitzer \cite{KKS75}, we see that
    $$
\lim_{n\to\infty} \frac{X_n\cdot e_1}{\sqrt{(X_n\cdot e_1)^2}}= 1
    $$
    in distribution, and hence also in $Q_1$- probability. It follows that
for each $\varepsilon>0$,
the left hand-side of (\ref{qiu1}) tends to $0$ as $n\to\infty$.

\smallskip
\noindent {\it Part (iv).} For $j\in\{1,2\}$ and $a$ a positive real number, we define the stopping times $T_{a}^{e_j}$ and $\widetilde{T}_{a}^{e_j}$ by
    \begin{equation}
    \label{st1}
    T_{a}^{e_j}:=\inf\{n\geq 0: X_n\cdot e_j\geq a\}
    \end{equation}
    along with
    \begin{equation}
    \label{st2}
    \widetilde{T}_{a}^{e_j}:=\inf\{n\geq 0: X_n \cdot e_j\leq a\}
    \end{equation}
Notice that for $c=1$ and  large $L$ one has the following estimate
\begin{equation}
\label{q1inequality}
Q_1[X_{T_{B_{L,cL,l}}(0)}\not\in \partial^+B_{L,cL,l}(0)]\leq Q_1[\widetilde{T}_{-L}^{e_1}<T_L^{e_1}]+Q_1[ T_{L}^{e_2}\wedge \widetilde{T}_{-L}^{e_2}< T_{L}^{e_1} ].
\end{equation}
The first probability in the right-most side of (\ref{q1inequality}) has an exponential bound in $L$. Observe that the second probability in the right-most side of (\ref{q1inequality}) is less than or equal to
$$
Q_1[ T_{L}^{e_2}\wedge \widetilde{T}_{-L}^{e_2}\leq L^{2+\varepsilon} ]+Q_1[ L^{2+\varepsilon}< T_{L}^{e_1} ].
$$
Keeping the notations introduced in item $(iii)$,
 one sees that for large $L$, there exists a positive constant $K_1$ such that
\begin{eqnarray}
\nonumber
&Q_1[ T_{L}^{e_2}\wedge \widetilde{T}_{-L}^{e_2}\leq L^{2+\varepsilon} ] \leq Q_1[|X_{n}\cdot e_2|\leq L, \mbox{ for all }n\in \mathbb{N}, 0\leq n\leq L^{2+\varepsilon}]\\
\label{estq1}
&\le\widetilde{P}[Z_{N_2(n)}\leq L^{2+\varepsilon}, \mbox{ for all }n\in \mathbb{N}, 0\leq n\leq L^{2+\varepsilon}]\leq \exp\{-K_1L^{\varepsilon}\}.
\end{eqnarray}
On the other hand, using the sharp estimate in Theorem 1.3 in \cite{FGP10} and denoting by $\hat{P}$ the law of underlying one-dimensional random walk corresponding to the annealed law of $(X_n \cdot e_1)_{n\geq 0}$, we can see that for large $L$, there exists a positive constant $K_2$ such that

\begin{eqnarray}
\nonumber
&Q_1[ L^{2+\varepsilon}< T_{L}^{e_1} ]\leq Q_1[ X_{[L^{2+\varepsilon}]}\cdot e_1< L]\\\
\label{estq2}
&\le\hat{P}[Y_{N_1([L^{2+\varepsilon}])}<L]\leq K_2 L^{-(\kappa-1/2-\varepsilon)}.
\end{eqnarray}
Therefore, in view of  inequality (\ref{q1inequality}),  the estimates (\ref{estq1})
and (\ref{estq2}), we complete the proof.
\end{proof}

\medskip

\subsection{Directionally transient random walk without an asymptotic direction}
 Let $\{p_i:i\in\mathbb Z\}$
and $\{p'_j:j\in\mathbb Z\}$
 be two independent i.i.d. copies
of $p$. Following a similar procedure as in the previous example, we consider
in  the lattice $\mathbb{Z}^2$ the canonical vectors $e_1$ and $e_2$, and define the random environment
$\omega=\{\omega((i,j)):(i,j)\in\mathbb Z^2\}$ by,

$$
\omega_{(i,j)}(e_1)=\frac{p_i}{2}\quad {\rm and}\quad
\omega_{(i,j)}(-e_1)=\frac{1}{2}-\frac{p_i}{2}.
$$
together with

$$
\omega_{(i,j)}(e_2)=\frac{p'_j}{2}\quad {\rm and}\quad
\omega_{(i,j)}(-e_2)=\frac{1}{2}-\frac{p'_j}{2}.
$$
We call $\mathbb P_2$ the law of the above environment
and $Q_2$ the annealed law of the corresponding random
walk starting from $0$.

\begin{theorem}
\label{ex2} Consider a random walk in a random environment
with law $\mathbb P_2$. Then, the following are satisfied.
\begin{itemize}
\item[$(i)$] Let $l\in\mathbb S$. Then
 $l\cdot e_1\geq0$ and $l\cdot e_2\geq0$
if and only if  $Q_2$-a.s.

$$
\lim_{n\to\infty}X_n\cdot l=\infty.
$$

\item[$(ii)$] $Q_2$-a.s.

$$
\lim_{n\to\infty}\frac{X_n}{n}=0.
$$

\item[$(iii)$] There exists a non-deterministic $\hat v$
such that

$$
\frac{X_n}{|X_n|_2}\rightarrow\hat v.
$$
in distribution.

\item[$(iv)$] There exists a $c>1$ such that

\begin{equation}
\label{t1p2}
\overline\lim_{L\to\infty}L^{-1}  \log Q_2[X_{T_{B_{L,cL,l}}(0)}\not\in \partial^+B_{L,cL,l}(0)]<0,
\end{equation}
where $l=(1/\sqrt{2}, 1/\sqrt{2} )$. Thus, condition $(T)|l$ \cite{Sz02}
is satisfied.
\end{itemize}

\end{theorem}

\begin{proof}
{\it Part (i).}
 It is enough to prove that $Q_2$-a.s.
$$
\lim_{n\to\infty}X_n\cdot e_1=\infty \mbox{ and }\lim_{n\to \infty}X_n\cdot e_2=\infty.
$$
Both assertions follow from an argument similar to the one used
in part $(i)$ of Theorem \ref{te1}, Theorem 2.1.2 in \cite{Z04} and  (\ref{tnb}).

\smallskip
\noindent {\it Part (ii).}
 This proof is similar to case $(ii)$ of Theorem \ref{te1} . 

\smallskip
\noindent {\it Part (iii).}
For $j=1,2$ we define $T_{0,j}=0$. For $j=1,2$ we define

$$
T_{1,j}=\inf\{n\geq0:\  (X_n-X_0)\cdot e_j>0 \mbox{ or } (X_n-X_0)\cdot e_j<0\}
$$
and for $i\geq2$ let
$$
T_{i,j}= T_{1,j}\circ \theta_{T_{i-1,j}}+T_{i-1,j}.
$$
Setting $Y_{n,j}:=X_{T_{n,j}}\cdot e_j$, we see that for $j\in\{1,2\}$, the one dimensional random walks without transitions to itself at each site $(Y_{n,j})_{n\geq0}$ are independent and their transitions at each site $i\in \mathbb{Z}^d$ are determined by $p_i$. Furthermore, for $j\in\{1,2\}$, the strong law of large numbers implies that $Q_2$- a.s.
\begin{equation}
\label{naiveslln}
\lim_{n\to\infty}\frac{T_{n,j}}{n}=2.
\end{equation}
We now apply the result of Kesten, Kozlov and Spitzer \cite{KKS75} to see that there exist constants $C_1$ and $C_2$ such that
$$
\left(\frac{Y_{n,1}}{n^\kappa},\frac{Y_{n,2}}{n^\kappa}\right)\rightarrow\left(C_1\left(\frac{1}{S_{ca}^{1 \ \kappa}}\right)^\kappa, C_2\left(\frac{1}{S_{ca}^{2 \ \kappa}}\right)^\kappa\right)
$$
in distribution, where for $j\in\{1,2\}$, $S_{ca}^{j \ \kappa}$ stands for two independent completely asymmetric stable laws of index $\kappa$, which are positive. Using (\ref{naiveslln}) and properties of convergence in distribution we can see that
$$
\frac{X_n}{|X_n|_2}=\frac{\frac{(X_n\cdot e_1)}{n^{\kappa}}e_1+\frac{(X_n\cdot e_2)}{n^{\kappa}}e_2}{\sqrt{\frac{(X_n\cdot e_1)^2}{n^{2\kappa}} +\frac{(X_n\cdot e_2)^2}{n^{2\kappa}}}}\rightarrow \frac{\left(\frac{C_1}{S_{ca}^{1 \ \kappa}}\right)^\kappa e_1+\left(\frac{C_2}{S_{ca}^{2 \ \kappa}}\right)^\kappa e_2}{\sqrt{\left(\frac{C_1}{S_{ca}^{1 \ \kappa}}\right)^{2\kappa}+ \left(\frac{C_2 }{S_{ca}^{2 \ \kappa}}\right)^{2\kappa}}}
$$
in distribution. Therefore we have proved that the limit $\hat{v}$ is random.

\smallskip
\noindent {\it Part (iv).}
 A first step will be to prove the following decay
$$
\limsup L^{-1} \log Q_2[\widetilde{T}_{-\widetilde{c}L}^{e_j}<T_{cL}^{e_j}]<0
$$
for arbitrary positive constants $\widetilde{c}$ and $c$ (see (\ref{st1}) and (\ref{st2}) for the notations). We will prove this only in the case $j=1$ since the case $j=2$ is
similar. Following the notation introduced in Theorem \ref{te1} item $(i)$ and denoting the greatest integer function by $[\cdot]$, we see that it is sufficient to prove that for large $L$ there exists a positive constant $\widehat{C}$ such that:
\begin{equation}
\label{claimex}
\widetilde{E}_1[P_{0,\omega}[Y_{n} \mbox{ hits } -[\widehat{c}L]+1 \mbox{ before } [cL]+1 ]]\leq \exp\{-\widehat{C}L\}.
\end{equation}
To this end, for a fixed random environment $\omega$, if we define
$$
\mathfrak{V}_i^{L}:=
P_{i,\omega}[Y_{n} \mbox{ hits } -[\widehat{c}L]+1 \mbox{ before } [cL]+1],
$$
the Markov property makes us see that $\mathfrak{V}_i^{L}$ satisfies the following difference equation for integers $i\in[[\widehat{c}L]+2,[cL]]$,
$$
\mathfrak{V}_i^{L}=(1-p_i)\mathfrak{V}_{i-1}^{L}+p_1\mathfrak{V}_{i+1}^{L},
$$
with the constraints
$$
\mathfrak{V}_{[\widetilde{c}L]+1}^L=1 \mbox{  and  } \mathfrak{V}_{[cL]+1}^L=0.
$$
This system can be solved by the method developed by Chung in \cite{Ch67}, Chapter 1, Section 12. Applying it we see that
$$
\mathfrak{V}_0^L=\frac{\exp\{\sum_{-[\widehat{c}L+1],0}\}+\ldots +\exp\{\sum_{-[\widehat{c}L]+1,[cL]}\}}{1+\exp\{\sum_{-[\widehat{c}L]+1,-[\widehat{c}L]+2}\}+\ldots+\exp\{\sum_{-[\widehat{c}L]+1, [cL]}\}},
$$
where we have adopted the notation  $\sum_{z<m\leq z'}:=\log \rho(m)$ and $\rho(m):=(1-p_m)/p_m$.
A slight variation of the argument in \cite{Sz02} page 744 completes the proof of claim (\ref{claimex}). On the other hand, considering the probability
$$
Q_2[X_{T_{B_{L,2L,l}}(0)}\not\in \partial^+B_{L,2L,l}(0)],
$$
we observe that this expression is clearly bounded from above by (see Figure \ref{slabex})
$$
Q_2[\widetilde{T}_{-\frac{\sqrt{2}}{2}L}^{e_1}<T_{\sqrt{2}L}^{e_1}]+Q_2[\widetilde{T}_{-\frac{\sqrt{2}}{2}L}^{e_2}<T_{\sqrt{2}L}^{e_2}]
$$
\begin{figure}[h]
  \begin{center}
  \includegraphics[width=6cm]{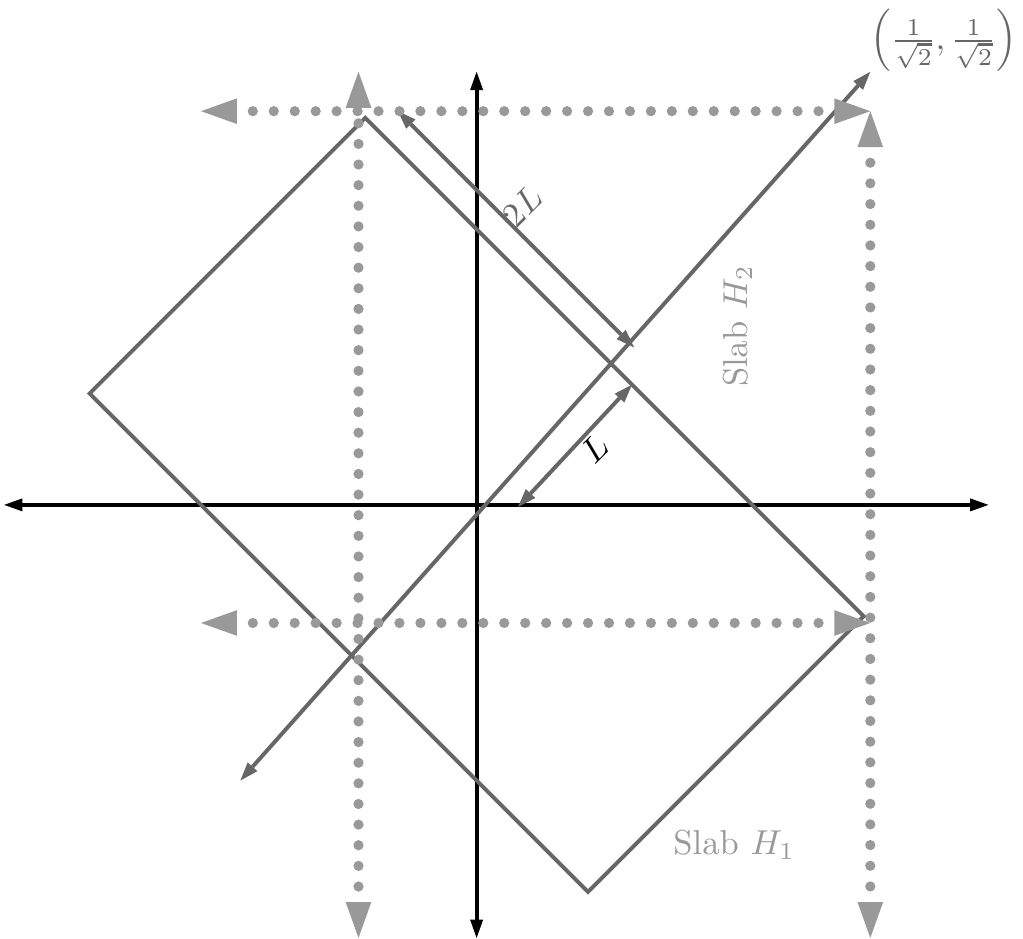}\\
  \caption{A geometric sketch of the bound for $Q_2[X_{T_{B_{L,2L,l}}(0)}  \notin \partial^+ B_{L,2L,l}(0)]$.}\label{slabex}
  \end{center}
\end{figure}
\noindent
In virtue of the claim (\ref{claimex}) the last expression has an exponential bound and this finishes the proof.
\end{proof}

\section{Preliminary discussion}
\label{pddis}

In this section we will  derive some important properties that are satisfied by
the non-effective polynomial and cone mixing conditions. In subsection
\ref{sub1} we will show that the non-effective polynomial condition is
weaker than the conditional form of Kalikow's condition introduced in
\cite{CZ02}. In subsection \ref{sub2} we will show that the cone mixing condition
implies ergodicity. Finally, in subsection \ref{sub3}, we will prove
that the non-effective polynomial condition in a given direction implies
the non-effective polynomial condition in a neighborhood of that direction
with a lower degree.

\medskip

 \subsection{Non-effective polynomial condition  and its relation with other directional
transience
conditions}
\label{sub1}
Here we will
 discuss the relationship between the condition non-effective polynomial
condition and
 other transience conditions.
Furthermore we will  show that the conditional non-effective polynomial condition is
weaker than the conditional version of Kalikow's condition introduced by
Comets-Zeitouni in \cite{CZ01} and \cite{CZ02}.

For reasons that will become clear in the next section, the following definition,
which is actually weaker than the conditional non-effective polynomial condition,
will be useful.
Let $l\in\mathbb S^{d-1}$, $M\ge 1$ and $c>0$. We say that condition $(P)_{M,c}|l$ is satisfied,
 and we call it the \textit{non-effective polynomial condition} if there
is a constant $c>0$ such that

$$
\overline\lim_{L\to\infty}L^M  P_0[X_{T_{B_{L,cL,l}}(0)}\not\in \partial^+B_{L,
{ c}L,l}(0)]= 0.
$$
It is straightforward to see that $(PC)_{M,c}|l$ implies $(P)_{M,c}|l$.

It should be pointed out, that for a fixed $\gamma\in (0,1)$,
 if both in the conditional and non-conditional
non-effective polynomial conditions the polynomial decay is replaced by a stronger stretched
exponential decay of the form $e^{-L^\gamma}$, one would
obtain a condition defined on rectangles equivalent to condition $(T)_\gamma$
introduced by Sznitman in \cite{Sz03}, and also a conditional version of it.
On the other hand, as we will see now, the conditional non-effective polynomial
condition is implied by  Kalikow's condition as defined
in \cite{CZ01} for environments which
are not necessarily i.i.d. Let us recall
this definition.
For $V$ a finite, connected subset of $\mathbb{Z}^d$, with $0\in V$ , we let
$$
\mathfrak{F}_{V^c}=\sigma\{\omega(z,\cdot):z\not \in V \}.
$$
The \textit{Kalikow's random walk} $\{X_n:n\geq 0\}$ with state space in $V\cup \partial V$, starting from $y\in V\cup \partial V$
is defined by the transition probabilities

$$
\widehat{P}_V(x,x+e):=\left\{
\begin{array}{ll}
\frac{E_0[\sum_{n=0}^{T_{V^c}}\mathds{1}_{\{X_n=x\}}\omega(x,e)|\mathfrak{F}_{V^c}]}
{E_0[\sum_{n=0}^{T_{V^c}}\mathds{1}_{\{X_n=x\}}|\mathfrak{F}_{V^c}]}
,&\quad{\rm for}\quad x\in V\ {\rm and}\ e\in E\\
1&\quad{\rm for}\quad x\in \partial V\ {\rm and}\  e=0.
\end{array}
\right.
$$
We denote by $\hat P_{y,V}$ the law of this random walk and by $\hat E_{y,V}$ the
corresponding expectation. The importance of Kalikow's random walk stems from the fact that

\begin{equation}
\label{distofXtu}
X_{T_{V^c}} \ \ \mbox{has the same law under }\widehat{P}_{0,V}\ \mbox{and under }P_0[\cdot|\mathfrak{F}_{V^c}]
\end{equation}
(see   (\cite {K81})).
Let $l\in\mathbb S^{d-1}$.
We  now define Kalikow's condition with respect to the  direction $l$
 as the following requirement:
there exits a positive constant $\delta$ such that

$$
\inf_{V: x\in V} \ \widehat{d}_{V}(x) \cdot l\geq\delta,
$$
where

$$
\widehat{d}_{V}(x):=\widehat{E}_{x,V}[X_1-X_0]=\sum_{e\in E}
e\widehat{P}_V(x,x+e)
$$
denotes the drift of Kalikow's random walk  at $x$, and the infimum runs over all finite connected
subset $V$ of $\mathbb{Z}^d$ such that $0\in V$.
The following result shows that Kalikow's condition is indeed stronger that the
conditional non-effective polynomial criteria.

\medskip

\begin{proposition}
\label{kit}
Let $l\in \mathbb{S}^{d-1}$. Assume Kalikow's condition with respect to $l$. Then
there exists an $r>0$ such that for all $y \in H_{0,l}$ one has that

\begin{eqnarray*}
\nonumber
&\!\!\!\!©∂\displaystyle\limsup_{\substack{L\to\infty}}L^{-1} \sup \log P_0[X_{T_{B_{L,rL,l}}(0)}\not\in \partial^+B_{L,
 rL,l}(0),
T_{B_{L,rL,l}(0)}<T_{H_{y,l}}|\mathcal H_{y,l}]\\
\nonumber 
&<0,
\end{eqnarray*}
where the supremum is taken in the same sense as in (\ref{Pl}). In particular, Kalikow's condition with respect to direction $l$
implies $(PC)_{M,r}|l$ for all $M>0$.
\end{proposition}
\begin{proof}
Suppose that Kalikow's condition is satisfied with constant $\delta>0$.
We will first assume that $y\cdot l\in (-L,0)$.
 Let $c>1$. For $y \in H_{0,l}$ and  $L\ge 1$ consider  the box
$$
V:=R\left([y\cdot l,L]\times\left(-\frac{c}{\delta}L,\frac{c}{\delta}L\right)^{d-1}
\right).
$$
Therefore, using
(\ref{distofXtu}) we find that
\begin{eqnarray}
\nonumber
&P_0[X_{T_{B_{L,\frac{c}{\delta}L,l}}(0)}\not\in \partial^+B_{L,
 cL,l}(0),
T_{B_{L,\frac{c}{\delta}L,l}(0)}<T_{H_{y,l}}|\mathfrak{F}_{V^c}]\\
&\leq
\nonumber
P_0[X_{T_{\widetilde{U}}}\cdot R(e_j)\ge \frac{c}{\delta}L \mbox{ for some }j \in [2,d] ,|X_{T_{\widetilde{U}}}\cdot l|< L | \mathfrak{F}_{V^c}]\\
\label{totfk}
&=
 \widehat{P}_{0,V}[X_{T_{\widetilde{U}}}\cdot R(e_j)\geq \frac{c}{\delta}L \mbox{ for some }j \in [2,d],|X_{T_{\widetilde{U}}}\cdot l|< L ].
\end{eqnarray}
Notice that on the set
$$
\{X_{T_{V}}\cdot R(e_j)\ge \frac{c}{\delta}L \mbox{ for some }j,X_{T_{V}}\cdot l< L\},
$$
one has $\widehat{P}_{0,V}$-a.s. that
$$
T_{V}\ge\left[\frac{cL}{\delta}\right].
$$
Thus, by means of the auxiliary martingale $\{M_n^{V}:n\ge 0\}$
  defined by

$$
M_n ^{V}:=X_n-X_0-\sum_{j=0}^{n-1}\widehat{d}_V (X_j),
$$
which has bounded increments (indeed bounded by $2$) we can see that on $\{T_{V}>\left[\frac{cL}{\delta}]\right\}$, we have that for $L$ large enough that

\begin{equation}
\label{martingale}
M_{[\frac{cL}{\delta}]} ^{V}\cdot l < L -\left(
\frac{cL} {\delta}-1\right)\delta=(1-c)L+\delta< \frac{(1-c)L}{2}
\end{equation}
$\widehat{P}_{0,V}$-a.s.  Now,
 it will be convenient at this point to recall Azuma's inequality
(see for example \cite{Sz01}) for martingales with increments bounded by $2$,

$$
\widehat{P}_{0,V}[M_n^V\cdot w >A]\leq\exp\left\{-\frac{A^2}{8n}\right\} \ \mbox{ for } \ A>0, \ n\geq0, \ |w|=1.
$$
 Using this inequality and (\ref{martingale}) we obtain that

\begin{eqnarray}
\nonumber
&\widehat{P}_{0, V}[X_{T_{\widetilde{U}}}\cdot R(e_j)>\frac{c}{\delta}L \mbox{ for some }j,X_{T_{V}}\cdot l\leq L ]\\
\nonumber
&\leq
\widehat{P}_{0, V}[T_{V}>\frac{cL}{\delta}]\\
\label{ikalikow}
&\leq
\widehat{P}_{0, V}[M_{[\frac{cL}{\delta}]} ^{V}\cdot (-l) >(c-1)L/2]
\leq\exp\{-c_1 L\},
\end{eqnarray}
for a suitable positive constant $c_1$.
Finally, coming back to (\ref{totfk}), we can then conclude that

$$
\displaystyle\limsup_{\substack{L\to\infty}}L^{-1} \sup \log P_0[X_{T_{B_{L,rL,l}}(0)}\not\in \partial^+B_{L,
 rL,l}(0),
T_{B_{L,rL,l}(0)}<T_{H_{y,l}}|\mathcal H_{y,l}]<0,
$$
where $r=\frac{c}{\delta}$.
Let us now assume that $y\cdot l\le -L$. By Lemma 1.1 in \cite{Sz01}
we know that there exists a
positive constant $\psi$ depending on $\delta$ such that for all $V$
finite connected subsets of $\mathbb{Z}^d$ with $0\in V$
$$
e^{-\psi X_n\cdot l}
$$
is a supermartingale with respect to the canonical filtration of the walk
under Kalikow's law $\widehat{P}_{0,V}$.
Thus, we have that

$$
\widehat{P}_{0,V}[X_{T_V}\cdot l\leq-L]\leq \exp\{- \psi L\}
$$
by means of the stopping time theorem applied at time $T_V$.
By an argument similar to the one developed for the case
$y\cdot l\in (-L,0)$,
we can finish the estimate in the case $y\cdot l\le L$.
\end{proof}
\subsection{Cone mixing and ergodicity}
\label{sub2}
The main objective in this section is to establish that any stationary probability measure $\mathbb{P}$ defined on the canonical $\sigma-$ algebra $\mathfrak{F}$, which
satisfies property $(CM)_{\phi,\alpha}|l$ is ergodic with respect to space-shifts.
We do not claim any originality about such an implication, but since
we where not able to find an adecuate reference, we have included the proof
of  it here
for completeness.

Let us recall that
a set $E\in \mathfrak{F}$ is an invariant set if 
$$
\theta^{-1}_x E:=E
$$
for all $x\in \mathbb{Z}^d$.
\begin{theorem}
Assume that the probability space $(\Omega, \mathfrak{F}, \mathbb{P})$ has the property  $(CM)_{\phi,\alpha}|l$ and is stationary, then the probability measure $\mathbb{P}$ is \emph{ergodic}, i.e. for  any invariant set $E\in \mathfrak{F}$ we have:
$$
\mathbb{P}[E]\in \{0,1\}.
$$
\end{theorem}
\begin{proof}
Let $E\in \mathfrak{F}$ be an invariant set. Note that  for each $\epsilon>0$ 
there exists a cylinder measurable set $A\in \mathfrak{F}$ such that

$$
\mathbb{P}[A\triangle E]<\epsilon.
$$
Since $A$ is a cylinder measurable set, it can be represented as

\begin{eqnarray*}
& A=\{\omega(x,\cdot): \ x\in F, F\subset\mathbb{Z}^d, \ |F|<\infty, \\
& \omega(x_i,\cdot)\in P_i, \mbox{ for $x_i\in F$}, \ P_i\in\mathcal{B}(\mathcal{P}_d) \},
\end{eqnarray*}
where  $\mathcal{B}(\mathcal{P}_d)$ stands for the Borel $\sigma-$algebra on the compact subset $\mathcal{P}_d$ of $\mathbb{R}^{2d}$. Choose  now $L$ such that
$$
\phi(L)<\epsilon.
$$
Plainly, for $L$ we can find an $x\in \mathbb{Z}^d$ such that $\theta_x A$ and $A$ are $L$ separated on cones with respect to direction $l$:
there exists $y\in \mathbb{Z}^d$ such that
$$
A\in \sigma\{\omega(z,\cdot):\ z\cdot l\leq y\cdot l-L\}
$$
along with
$$
\theta_x A \in \sigma\{\omega(z,\cdot):\ z\in C(y, l, \alpha)\}.
$$
We can suppose that $\mathbb{P} [E]>0$, otherwise there is nothing to prove. So as to complete the proof we have to show that $\mathbb{P}[E]=1$. Therefore taking $\epsilon$ small enough we can suppose that $\mathbb{P}[A]>0$. Thus, using the cone mixing property,  we get that

\begin{equation}
\label{dif1}
-\mathbb{P}[A]\phi(L)\leq\mathbb{P}[A\cap(\theta_x A)^c ]-\mathbb{P}[A]\mathbb{P}[\Omega-A]\leq\mathbb{P}[A]\phi(L).
\end{equation}
On the other hand, since $E$ is an invariant set, we see that

\begin{equation}
\label{dif2}
\mathbb{P}[\theta_x A\triangle E]=\mathbb{P}[\theta_x A\triangle \theta_x E]=\mathbb{P}[\theta_x(A\triangle E)]<\epsilon,
\end{equation}
which implies 
\begin{equation}
\label{dif3}
\mathbb{P}[A\triangle \theta_x A]\leq \mathbb{P}[(A \triangle E)\cup (\theta_x A\triangle\theta_x E)]<2\epsilon.
\end{equation}
In turn, from inequality (\ref{dif3}), it is clear that $\mathbb{P}[A\cap(\theta_x A)^c ]<2\epsilon$. Now, using the inequality (\ref{dif1}) one has that
$$
\mathbb{P}[A]\mathbb{P}[\Omega-A]\leq 2\epsilon + \mathbb{P}[A]\phi(L).
$$
As a result, we see that
\begin{eqnarray}
\nonumber
&\mathbb{P}[E]\mathbb{P}[\Omega-E]<(\mathbb{P}[A]+\epsilon)(\mathbb{P}[\Omega-A]+\epsilon)\\
\nonumber
&=\mathbb{P}[A]\mathbb{P}[\Omega-A]+\epsilon+\epsilon^2
< 4\epsilon + \phi(L)\leq 5\epsilon.
\end{eqnarray}
 Hence, since $\epsilon>0$ is arbitrary we conclude that
 $\mathbb{P}[E]\mathbb{P}[\Omega-E]=0$. Therefore if $\mathbb{P}[E]>0$,
this implies $\mathbb{P}[E]=1$.
\end{proof}

\subsection{Polynomial Decay implies Polynomial decay in a neighborhood}
\label{sub3}

In this subsection we prove that whenever $(PC)_{M,c}|l$ holds, for prescribed positive constants $M$ and $c$, then we can choose
$2(d-1)$ directions where we still have polynomial decay although of less order.
More precisely, we can prove the following.
\begin{proposition}
\label{lneighbors}
Suppose that $(P)_{M,c}|l$ is satisfied with $c>0$ for some $M>6(d-1)$, then there exists an $\alpha>0$ such that if we define for $i\in[2,d]$,
$$
l_{+ i}:=\frac{l + \alpha R(e_i)}{|l + \alpha R(e_i)|}
$$
and
$$
l_{- i}:=\frac{l - \alpha R(e_i)}{|l - \alpha R(e_i)|},
$$
then
$$
(P)_{N,2c}|l_{\pm i}
$$
is satisfied with $N=\frac{M}{3}-1$.
\end{proposition}
\begin{proof}[Proof of Proposition \ref{lneighbors}]
We will just give the proof for direction $l_{-2}$, the other cases being analogous.
Throughout the proof we pick $\alpha\in(0,1)$ and we define the angle 
\begin{equation}
\label{def-beta}
\beta:=\arctan(\alpha).
\end{equation}
Consider the  rotation $R''$ on $\mathbb{R}^d$ defined by
$$
R'':=\left(
       \begin{array}{cccccc}
         \cos(\beta)& -\sin(\beta) & 0 & \ldots & \ldots &0\\
         \sin(\beta) & \cos(\beta) & 0 & \ldots & \ldots &0\\
         0 & 0& 1 & \ldots & \ldots&0\\
          \vdots & \vdots & \vdots  & \vdots   &\vdots& \vdots\\
          \vdots &\vdots  & \vdots  & \vdots  &\vdots &\vdots\\
         0 & 0 & \ldots &\ldots  & 1&0\\
         0 & 0 & \ldots &\ldots & 0&1 \\
       \end{array}
     \right).
$$
where this representation matrix is taken in the vector space base $\{R(e_1), R(e_2),\ldots,R(e_d)\}$.
It will be useful to define a new rotation
$$
R':=R''R
$$
together with the \emph{rotated box} $\widetilde{B}_{L}(0)$ given by
\begin{equation*}
\widetilde{B}_{L}(0):=R'\left(\left[-L\lambda_1(\alpha),L\lambda_2(\alpha)\right]\times \left[-L c \lambda_3(\alpha),L c \lambda_3(\alpha)\right]^{d-1}\right)\cap\mathbb{Z}^d,
\end{equation*}
where
\begin{eqnarray*}
&\lambda_1(\alpha):=\frac{1+\frac{1}{\alpha}}{\sqrt{\cot^2(\beta)+1}}\\
&\lambda_2(\alpha):=\frac{\frac{1}{\alpha}-1}{\sqrt{\cot^2(\beta)+1}}\\
&\lambda_3(\alpha):=\frac{\sqrt{(1-\cot(\beta))^2+(1-\tan(\beta))^2}}{|\tan(\beta)+\cot(\beta)|}
\end{eqnarray*}
Notice that with these definitions, $P_0$- almost surely:
\begin{equation}
\label{contained}
X_{T_{\widetilde{B}_{L}(0)}}\not\in \partial^+\widetilde{B}_{L}(0) \Rightarrow X_{T_{B_{L,L,l}(0)}}\not \in \partial^+B_{L,L,l}(0).
\end{equation}
Figure \ref{rb} shows the boxes involved in (\ref{contained}).
\begin{figure}[h]
  \begin{center}
  \includegraphics[width=5cm]{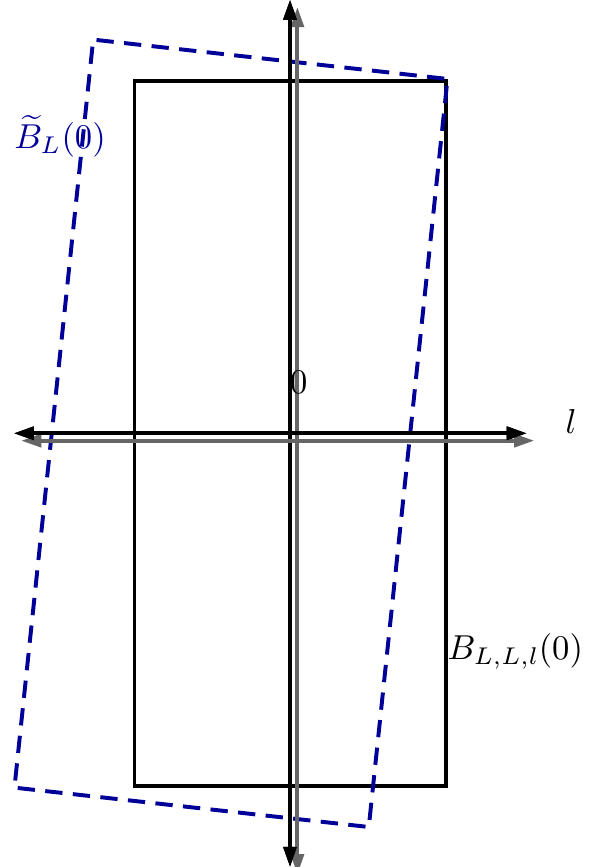}\\
  \caption{The choice of boxes.}\label{rb}
  \end{center}
\end{figure}

\medskip
\noindent
As a result we have that
$$
P_0[X_{T_{\widetilde{B}_{L}(0)}}\not\in \partial^+\widetilde{B}_{L}(0)]\leq L^{-M}.
$$
Furthermore, a straightforward computation makes us see that the scale factor $\lambda_3(\alpha)$ is less than $\frac{4}{3}$ whenever $\alpha\leq\frac{1}{9}$. Therefore if we let the positive $\alpha\leq\frac{1}{9}$ one has that
\begin{equation}
\label{ancho}
\lambda_3(\alpha)\leq \frac{4}{3}.
\end{equation}
For technical reasons, we need to introduce an auxiliary box. Specifically, we first set
$$
h:=\frac{\frac{1}{\alpha}-1}{\sqrt{1+(\frac{1}{\alpha})^2}}=  \frac{1-\alpha}{\sqrt{1+\alpha^2}}
$$
and observe that $\frac{4}{5}<h<1$. We then can introduce the new box $\overline{B}_{l_{-2},L}(0)$ defined by
\begin{equation*}
\overline{B}_{l_{-2},L}(0):=R'\left(\left[-L(h+2),Lh\right]\times \left[-cL\lambda_3(\alpha),Lc\lambda_3(\alpha)\right]^{d-1}\right)\cap\mathbb{Z}^d.
\end{equation*}
From this definition, we obtain
$$
P_0[X_{T_{\overline{B}_{l_{-2},L}(0)}}\not\in \partial^+\overline{B}_{l_{-2},L}(0)]\leq L^{-M}.
$$
In order to complete the proof, we claim that for large enough $U$ the probability
$$
P_0[X_{T_{\overline{B}_{l_{-2},U}(0)}}\not\in \partial^+\overline{B}_{l_{-2},U}(0)],
$$
decays polynomially as a function of $U$, where the box $\overline{B}_{l_{-2},U}(0)$ is defined by
$$
\overline{B}_{l_{-2},U}(0):=R'([-U,U]\times[-2cU, 2cU]^{d-1}).
$$
The general strategy to follow will be to stack smaller boxes up inside of $\overline{B}_{l_{-2},U}(0)$ and then using the  Markov property along with  \emph{good environment sets} we ensure that the walk exits from box $\overline{B}_{l_{-2},U}(0)$ by $\partial^+\overline{B}_{l_{-2},U}(0)$ with probability bigger than $1-P(U)$, where $P$ is a polynomial function. Specifically, we let
\begin{equation}
\label{L}
L:=\frac{U}{h+2}.
\end{equation}
Introduce now a sequence of stopping times as follows
$$
T_1=T_{\overline{B}_{l_{-2},L}(0)},
$$
and for $i>1$
$$
T_i=T_{i-1}+T_1\circ\theta_{T_{i-1}}.
$$
For simplicity we write $\widehat{T_1}$ instead of $T_{\overline{B}_{l_{-2},U}(0)}.$
In view of (\ref{ancho}) and (\ref{L}) it is clear that to ensure that the random walk
exits at time $\widehat T_1$ through $\partial^+\bar B_{l_{-2},U}(0)$, it is enough
that it exits through the corresponding positive boundaries through four
succesive times, so that
\begin{eqnarray}
\nonumber
&P_0[X_{\widehat{T_1}}\in \partial^+\overline{B}_{l_{-2},U}(0)]\geq P_0\left[X_{T_1}\in\partial^+\overline{B}_{l_{-2},L}(0),\right.\\
\nonumber
&\left.\left(X_{T_{1}}\in \partial^+\overline{B}_{l_{-2},L}(X_{T_1})\right)\circ\theta_{T_1},\left(X_{T_1}\in \partial^+\overline{B}_{l_{-2},L}(X_{T_2})\right)\circ\theta_{T_2},\right.\\
\label{inpol}
&\left. \left(X_{T_1}\in \partial^+\overline{B}_{l_{-2},L}(X_{T_{3}})\right)\circ\theta_{T_{3}}\right].
\end{eqnarray}
In order to use (\ref{inpol}), let $i$ be a positive integer number and consider the lattice sets sequence $(F_i)_{i\geq 1}$ defined by
$$
F_1=\partial^+\overline{B}_{l_{-2},L}(0),
$$
and for $i>1$, we define by induction
$$
F_i=\bigcup_{y\in F_1}\partial^+\overline{B}_{l_{-2},L}(y).
$$
We now define for $i\geq1$, the environment events $G_i$ by
\begin{eqnarray*}
&G_i=
\left\{\omega\in \Omega:P_{y,\omega}[\left(X_{T_1}\in \partial^+\overline{B}_{l_{-2},L}(X_{T_i})\right)\circ\theta_{T_i}]\right. \\
&\left. \geq1-L^{-\frac{M}{2}}, \mbox{ for each } y \in F_i\right\}.
\end{eqnarray*}
Note that
\begin{eqnarray*}
&P_0[X_{\widehat{T_1}}\in \partial^+\overline{B}_{l_{-2},U}(0)] \\
&\ge P_0\left[X_{T_1}\in\partial^+\overline{B}_{l_{-2},L}(0),\left(X_{T_{1}}\in \partial^+\overline{B}_{l_{-2},L}(X_{T_1})\right)\circ\theta_{T_1},\right.\\
&\left.\left(X_{T_1}\in \partial^+\overline{B}_{l_{-2},L}(X_{T_2})\right)\circ\theta_{T_2},\left(X_{T_1}\in \partial^+\overline{B}_{l_{-2},L}(X_{T_{3}})\right)\circ\theta_{T_{3}}\mathds{1}_{G_3}\right].
\end{eqnarray*}
By the Markov property applied at time $T_3$ and the very meaning of $G_3$, we get that the last expression is equal to
\begin{eqnarray}
\nonumber
&\sum_{y\in F_3}\mathbb{E}\left[P_{0,\omega}\left[X_{T_1}\in\partial^+\overline{B}_{l_{-2},L}(0),\left(X_{T_{1}}\in \partial^+\overline{B}_{l_{-2},L}(X_{T_1})\right)\circ\theta_{T_1},\right.\right.\\
\nonumber
&\left.\left.\left(X_{T_{1}}\in \partial^+\overline{B}_{l_{-2},L}(X_{T_2})\right)\circ\theta_{T_2}\right]P_{y,\omega}[X_{T_{\overline{B}_{l_{-2},L}(y)}}\in
\partial^+\overline{B}_{l_{-2},L}(y)]\mathds{1}_{G_3}\right]\geq\\
\nonumber
&(1-L^{-\frac{M}{2}})\left(P_0\left[X_{T_1}\in\partial^+\overline{B}_{l_{-2},L}(0),\left(X_{T_{1}}\in \partial^+\overline{B}_{l_{-2},L}(X_{T_1})\right)\circ\theta_{T_1},\right. \right.\\
\label{g33}
&\left.\left.\left(X_{T_{1}}\in \partial^+\overline{B}_{l_{-2},L}(X_{T_2})\right)\circ\theta_{T_2}\right]-\mathbb{P}[(G_3)^c]\right).
\end{eqnarray}
Repeating the above argument, one has the following upper bound for the right-most expression of (\ref{g33}),
\begin{equation}
\label{cubic}
(1-L^{-\frac{M}{2}})^4-(1-L^{-\frac{M}{2}})^3\mathbb{P}[(G_1)^c]-(1-L^{-\frac{M}{2}})^2\mathbb{P}[(G_2)^c]-(1-L^{-\frac{M}{2}})\mathbb{P}[(G_3)^c].
\end{equation}
At this point, we would like to obtain for $i\in|[1,3]|$,  an upper bound of the probabilities
$$
\mathbb{P}[(G_i)^c].
$$
To this end, we first observe that  Chebyshev's inequality and our hypothesis imply
\begin{equation*}
\mathbb{P}[(G_1)^c]\leq \sum_{y\in F_1}\mathbb{E}[\mathds{1}_{\{P_{y,\omega}[X_{T_{\overline{B}_{l_{-2},L}(y)}}\in\partial^+\overline{B}_{l_{-2},L}(y)]>L^{-\frac{M}{2}}\}}]\leq
\mid F_1\mid L^{-\frac{M}{2}}.
\end{equation*}
Clearly, we have the estimate $\mid F_1\mid \leq \left(\frac{8}{3}L\right)^{d-1}$
(recall (\ref{ancho})). As a result, we see that
\begin{equation}
\label{g1}
\mathbb{P}[(G_1)^c]\leq \left(\frac{8}{3}L\right)^{d-1} L^{-\frac{M}{2}}.
\end{equation}
By a similar procedure we can conclude that
\begin{equation}
\label{g2}
\mathbb{P}[(G_2)^c]\leq\left(\frac{16}{3}L\right)^{d-1}L^{-\frac{M}{2}}\quad
{\rm and}\quad \mathbb{P}[(G_3)^c]\leq\left(\frac{24}{3}L\right)^{d-1}L^{-\frac{M}{2}}.
\end{equation}
Combining the estimates in (\ref{inpol}) and (\ref{g2}) 
and the assumption $M\geq6(d-1)$ we see that

\begin{equation*}
P_0[X_{\widehat{T_1}}\not \in \partial^+\overline{B}_{l_{-2},U}(0)]\leq3 \frac{6(8)^{d-1}}{2^{-\frac{M}{3}}}U^{-\frac{M}{3}}.
\end{equation*}
This ends the proof by choosing the required $\alpha$ as any number in the open interval $(0,\frac{1}{9})$.
\end{proof}

\medskip

\section{Backtracking of the random walk out of a cone}
\label{pdd}

Here we will provide a uniform control on the probability that
 a random walk starting form the vertex of a cone stays inside
the cone forever. It will be useful to this end to define

\begin{equation}
\label{defdp}
D':=\inf\{n \in \mathbb{N}:X_n \not \in \mathcal{C}(\alpha, l, X_0)\},
\end{equation}
where as before $l\in\mathbb S^{d-1}$.

\medskip

\begin{proposition}
\label{D'}
 Let $l\in\mathbb S^{d-1}$.
Suppose that $(P)_{M,c}|l$ holds, for some $M>6d-3$.
Then there exists a positive constant $c_2(d)>0$ such that $P_{0} [ D'= \infty ]>c_2(d)$.
\end{proposition}
In what follows we prove  this proposition. With the purpose
of making easier the reading, we introduce here  some notations.
Let $l'\in\mathbb S^{d-1}$ and choose a rotation $R'$ on $\mathbb{R}^d$ with the property
$$R'(e_1)=l'$$
\medskip

\noindent For each $x\in\mathbb Z^d$, real numbers $m>0$, $c>0$ and integer $i\ge 0$ we define the box

\begin{eqnarray*}
&B_i(x):=\\
&x+R'\left(\left(-2^{m+i},2^{m+i}\right)\times\left(-2c2^{m+i},2c2^{m+i}\right)^{d-1}\right)\cap\mathbb{Z}^d
\end{eqnarray*}
along with its "positive boundary"

$$
\partial^{+} B_i(x):= \partial B_i(x)\cap \{x+R'\left( (2^{m+i},\infty)\times\mathbb R^{d-1}
\right)\}.
$$
We also need slabs perpendicular to direction $l'$. Set
$$
V_0(x):=x+R'\left([-2^m,2^m]\times\mathbb R^{d-1}\right)\cap\mathbb{Z}^d
$$
and  for $i\geq1$,

$$
V_i(x):=x+R'\left(\left[-2^{m}, \sum_{j=0}^{i}2^{m+j}\right]\times\mathbb R^{d-1}\right)\cap\mathbb{Z}^d.
$$
The positive part of the boundary for this set is defined as

$$
\partial^+ V_i(x):=\partial V_i(x) \cap \left\{x+ R' \left(
\left( \sum_{j=0}^{i}2^{m+j},\infty\right)\times\mathbb R^{d-1}\right)\right\}.
$$
Furthermore, we will define recursively a sequence of stopping times as follows.
First, let

$$
T_0:=T_{B_0(X_0)}.
$$
and for $i\geq1$

$$
T_i:=T_{B_i(X_{ T_{i-1}})}\circ \theta_{ T_{i-1} } + T_{i-1}.
$$
We now need to define the first time of entrance of the random walk
 to
the hyperplane $R'\left((-\infty,0)\times\mathbb R^{d-1}\right)$,
$$
D_{l'}:=\inf\{n \ge 0:X_n\cdot l' < 0\}.
$$
With these notations we can prove:

\begin{lemma}
\label{lemmaD'}
Assume  $(P)_{N,2c}|l'$ where $c>0$ , for some $N> 2(d-1)$.
Then, for all $m\in\mathbb N$ and $x \in \{z\in\mathbb Z^d: z\cdot l'\geq 2^m\}$, we have that

$$
P_x[D_{l'}=\infty]\geq y(m)
$$
where $y(m)$ does not depend on $l'$ and satisfies $\lim_{m\rightarrow\infty}y(m)=1$.
\end{lemma}

\begin{proof}
From the fact that $(P)_{N,2}|l'$ holds, we can (and we do) assume that there
exists a $m>0$ large enough, such that for any positive integer $i$ one has that
\begin{equation}
\label{t1}
P_0[X_{T_{B_i (0)}}\in \partial^+ B_i(0)]\geq 1-2^{-N(m+i)}
\end{equation}
holds. By stationarity, we have for $x\in \mathbb{Z}^{d}$:
\begin{equation}
\label{t1x}
P_{x}[X_{T_{B_i (x)}}\in \partial^+ B_i(x)]\geq 1-2^{-N(m+i)} .
\end{equation}
Throughout this proof, let us choose $x\in\{z\in\mathbb Z^d:z\cdot l'\ge 2^m\}$.
For reasons that will be clear through the proof, we need to estimate for $i\ge 1$ the following probability

\begin{equation}
\label{Ii0}
I_i:=P_{x}[X_{T_{V_i (x)}}\in\partial^+ V_i(x)],
\end{equation}
and with this aim, in view of (\ref{t1x}), we have
$$
I_0 \geq P_x[X_{T_{B_0 (x)}}\in \partial^+ B_0(x)] \geq 1-2^{-Nm}\geq  1-2^{-N\frac{m}{2}}.
$$
Now, as a preliminary computation for the recursion, we begin to estimate $I_1$. Note that

\begin{equation}
\label{I1in1}
I_1 \geq P_ x[ X_{T_{0}}\in \partial^+ B_0 (X_0) , ( X_{T_{B_1(X_0)}} \in \partial^+ B_1 (X_0))\circ \theta_{T_0}].
\end{equation}
Using the strong Markov property at time $T_0$ we then see that

\begin{eqnarray}
\nonumber
&I_1 \ge \sum_{y\in\partial^+B_0(x)}\mathbb{E}\left[ P_{x,\omega}[X_{T_{0}}\in \partial^+ B_0 (X_0) , X_{T_{0}}=y] \right.\\
\nonumber
& \times \left.P_{y,\omega}[X_{T_{B_1(y)}} \in \partial^+ B_1 (y)] \right]\\
\nonumber
 &\geq \sum_{y\in\partial^+B_0(x)}\mathbb{E}\left[ P_{x,\omega}[X_{T_{0}}\in \partial^+ B_0 (X_0) , X_{T_{0}}=y]\right.\\
\label{I1in2}
&\times\left.P_{y,\omega}[X_{T_{B_1(y)}} \in \partial^+ B_1 (y)]\mathds{1}_{G_0} \right],
\end{eqnarray}
where

\begin{eqnarray*}
&G_0:=\\
&\{w\in\Omega:P_{y,\omega}[X_{T_{B_1(y)}}\in \partial^+ B_1 (y)]>1-2^{-N\frac{m}{2}},\ \mbox{for all}\ y \in \partial^+ B_0 (x)\}.
\end{eqnarray*}
Thus, it is clear that
\begin{equation}
\label{I1in3}
I_1\geq\left(1-2^{-N\frac{m}{2}}\right)\left(P_{x}[X_{T_{0}}\in \partial^+ B_0 (X_0)]-
\mathbb P[(G_0)^c]\right).
\end{equation}
Notice that by (\ref{t1x}) and Chebyshev's inequality
\begin{eqnarray}
\nonumber
&\mathbb P[ (G_0)^{c}] \leq \sum_{y\in\partial^+B_0(x)}\mathbb{P}[ P_{y,\omega}[X_{T_{B_1(y)}}\not\in\partial^+ B_1 (y) ]\geq 2^{-N\frac{m}{2}}]\\
\nonumber
&\leq \sum_{y\in\partial^+B_0(x)}P_y[X_{T_{B_1(y)}}\not\in \partial^+ B_1 (y) ]2^{N\frac{m}{2}}\\
\nonumber
&= |\partial^+B_0(x)|2^{N\frac{m}{2}}]P_0[X_{T_{B_1(0)}}\not\in \partial^+ B_1 (0)]\\
\label{G0in}
&\leq (2c 2^{m+1})^{d-1}2^{N\left(\frac{m}{2}-(m+1)\right)}\leq (2c 2^{m+1})^{d-1}2^{-N\frac{m}{2}}.
\end{eqnarray}
Plugging (\ref{G0in}) into (\ref{I1in3}) we see that

\begin{equation}
\label{ione}
I_1\geq(1-2^{-N\frac{m}{2}})(1-2^{-N\frac{m}{2}}-(2c 2^{m+1})^{d-1}2^{-N\frac{m}{2}}).
\end{equation}
 Hereafter we can do the general recursive procedure. For this end, we  define for $i\geq1$

\begin{equation}
\label{Ji}
\begin{split}
J_i:=P_0[X_{T_{0}}\in \partial^+ B_0 (X_0), ( X_{T_{B_1(X_0)}} \in \partial^+ B_1 (X_0))\circ \theta_{T_0},\ldots\\
\ldots ,( X_{T_{B_i(X_0)}} \in \partial^+ B_i (X_0))\circ \theta_{T_{i-1}}].
\end{split}
\end{equation}
It is straightforward that $I_i\geq J_i$.
Furthermore, through induction on $i\ge 1$, we will establish the following claim
\begin{equation}
\label{claim}
J_i\geq(1-2^{-N\frac{(m+i-1)}{2}})\left[J_{i-1}-2^{-N\frac{(m+i-1)}{2}}
\left(\sum_{j=0}^{i-1}2c 2^{(m+j)+1}\right)^{d-1} \right].
\end{equation}
To prove this, we first define the \emph{extended boundary}
 of the pile of boxes at a given step as

$$
F_0:=\partial B_0(x)\cap \{x+R'((2^m,\infty)\times\mathbb R^{d-1}))\},
$$
and for $i\ge 2$

$$
F_{i-1}:=\partial\left\{\cup_{y\in F_{i-2}}B_{i-1}(y)\right\}\cap
\{x+R'((2^{m+i-1},\infty)\times\mathbb R^{d-1}))\}.
$$
Using these notations, we can apply the strong Markov property to (\ref{Ji}) at time $T_{i-1}$,
to get that
\begin{eqnarray*}
&J_i= \sum_{y \in F_{i-1}}\mathbb{E} \left[ P_{x,\omega}[X_{T_{0}}\in \partial^+ B_0 (X_0) , \ldots\right.\\
&\left.\ldots, ( X_{T_{B_{i-1}(X_0)}} \in \partial^+ B_{i-1} (X_0))\circ \theta_{T_{i-2}}, X_{T_{i-1}}=y]
 P_{y,\omega}[ X_{T_{B_i(X_0)}} \in \partial^+ B_i(X_0)]\right].
\end{eqnarray*}
Following the same strategy used to deduce (\ref{ione}), it will be convenient to introduce
for each $i\ge 2$ the event
\begin{eqnarray*}
&G_{i-1}:=\\
&\{\omega\in \Omega : P_{y, \omega}[X_{T_{B_i(y)}} \in \partial^+ B_i(y)]> 1- 2^{-N\frac{(m+i-1)}{2}},\ \mbox{for all}\ y \in F_{i-1}\}.
\label{Gi}
\end{eqnarray*}
Inserting the indicator function of the event $G_{i-1}$ into (\ref{Ji}) we get that

\begin{eqnarray*}
&J_i\ge \\
&\sum_{y \in F_{i-1}}\mathbb{E} \left[ P_{x,\omega}[X_{T_{0}}\in \partial^+ B_0 (X_0) , \ldots,( X_{T_{B_{i-1}(X_0)}} \in \partial^+ B_{i-1} (X_0))\circ \theta_{T_{i-2}}, X_{T_{i-1}}=y]
\right.\\
&\left.\times P_{y,\omega}[ X_{T_{B_i(X_0)}} \in \partial^+ B_i(X_0)]\mathds{1}_{G_{i-1}}\right].
\end{eqnarray*}
By the same kind of estimation as in (\ref{I1in3}), we have
\begin{equation}
\label{Jiin1}
J_i\geq(1- 2^{-N \frac{(m+i-1)}{2}})\left(J_{i-1}-\mathbb P [(G_{i-1})^c] \right).
\end{equation}
We need to get an estimate for $\mathbb P[(G_{i-1})^c]$. We do it repeating
 the argument given in (\ref{G0in}).
Let us first remark that
\begin{equation}
\label{Fi}
|F_{i-1}|\leq \left(\sum_{j=0}^{i-1}2c 2^{(m+j)+1}\right)^{d-1},
\end{equation}
holds. Indeed, the case in which $l'=e_1$ gives the maximum number for $|F_{i-1}|$.
Keeping (\ref{Fi}) in mind we get that
\begin{eqnarray}
\nonumber
&P_{x}[(G_{i-1})^{c}] \leq \sum_{y \in F_{i-1}}\mathbb{P}\left[ P_{y,\omega}\left[X_{T_{B_i(y)}}\not\in\partial^+ B_i (y) \right]\geq 2^{-N\frac{(m+i-1)}{2}}\right]\\
\nonumber
&\leq \sum_{y\in F_{i-1}}P_y[X_{T_{B_i(y)}}\not\in \partial^+ B_i (y) ]2^{N\frac{(m+i-1)}{2}}\\
\label{Giin}
&\leq \left(\sum_{j=0}^{i-1}2c2^{(m+j)+1}\right)^{d-1}2^{-N\frac{(m+i-1)}{2}}.
\end{eqnarray}
Therefore, combining (\ref{Giin}) and (\ref{Jiin1}) we prove claim (\ref{claim}).
Iterating (\ref{claim}) backward, from a given integer $i$, we have got
\begin{equation}
\label{Jiin}
J_i\geq J_1\left[ \prod_{h=1}^{i-1}(1-2^{-N\frac{(m+h)}{2}})\right]-\sum_{j= 1}^{i-1}a_j 2^{-N\frac{m+j}{2}}\prod_{k=j}^{i-1}(1-2^{-N\frac{(m+k)}{2}}),
\end{equation}
where we have used for short
$$
a_j:=\left(\sum_{i=0}^{j}c2^{(m+i)+1}\right)^{d-1}\le (2c)^{d-1}2^{(m+j+2)(d-1)}.
$$
The same argument used to derive (\ref{ione}) can be repeated to
 conclude that
\begin{equation}
\label{J1in}
J_1\geq(1-2^{-N\frac{m}{2}})(1-2^{-N\frac{m}{2}}-(2c 2^{m+1})^{d-1}2^{-N\frac{m}{2}}).
\end{equation}
Replacing the right hand side of (\ref{J1in}) into (\ref{Jiin}), and together
to the fact $I_i\geq J_i$, we see that
\begin{equation}
\label{Ii}
I_i\geq \left[ \prod_{h=0}^{i-1}(1-2^{-N\frac{m+h}{2}})\right](1-2^{-N\frac{m}{2}})-\sum_{j= 0}^{i-1}a_j 2^{-N\frac{(m+j)}{2}}\prod_{k=j}^{i-1}(1-2^{-N\frac{(m+k)}{2}}).
\end{equation}
Now we can finish the proof. First, observe that
$$
P_{x}[D_{l'}=\infty]\geq I_\infty,
$$
where as a matter of definition
$$
I_\infty:=\lim_{i\rightarrow\infty}\, I_i
$$
(this limit exists, because it is the limit of a decreasing sequence of real numbers bounded from below). By the condition  $N> 2(d-1)$, we get that for each $m\ge 1$ one has that for
all $j\ge 1$,

$$
a_j \ 2^{-\frac{M(m+j)}{2}}\leq (8c)^{d-1} 2^{-\vartheta \frac{(m+j)}{2}},
$$
where  $\vartheta$ stands for the positive number so that $N=2(d-1)+\vartheta$. Thus
all the products and series in (\ref{Ii}) converge and we have that
for all $m\ge 1$ and $x\in \{z\in\mathbb Z^d: z\cdot l'\ge 2^m\}$

$$
P_{x}[D_{l'}=\infty]\geq y(m),
$$
where

\begin{eqnarray*}
&y(m):=
\left[ \prod_{h=0}^{\infty}(1-2^{-N\frac{(m+h)}{2}})\right](1-2^{-N\frac{m}{2}})\\
&-\sum_{j= 0}^{\infty}a_j 2^{-N\frac{(m+j)}{2}}\prod_{k=j}^{\infty}(1- 2^{-N\frac{(m+k)}{2}}).
\end{eqnarray*}
Clearly for each $m\ge 1$, $y(m)$ does not depend on the
direction $l'$ and $\lim_{m\to\infty}y(m)=1$, which completes the proof.
\end{proof}

\medskip
With the previous Lemma, we now have  enough tools to prove  Proposition \ref{D'}. Before this, we need a definition of geometrical nature.

We will say that a sequence $(x_0,\ldots, x_n)$ of lattice points is a \textit{path} if
for every $1\le i\le n-1$, one has that $x_i$ and $x_{i-1}$ are nearest neighbors.
Furthermore, we say that this path is \textit{admissible} if for every $1\le i\le n-1$
one has that
$$
(x_i-x_{i-1})\cdot l\ne 0.
$$

\medskip

\begin{proof}[Proof of Proposition \ref{D'}.] Assume $(P)_{M,c}|l$, where $M > 6(d-1)+3$ which
is  the condition of
the statement of the Proposition \ref{D'}.
We appeal to Proposition (\ref{lneighbors}) and assumption $(P)_{M,c}|l$ to choose an $\alpha>0$
such that for all $i\in [2,d]$
$$
(P)_{N,2c}|l_{\pm i}
$$
is satisfied with

\begin{equation}
\label{N}
N:=\frac{M}{3}-1>2(d-1).
\end{equation}
From now on,
let $m$ be any natural number satisfying

\begin{equation}
\label{ym}
y(m)> 1-\frac{1}{2(d-1)},
\end{equation}
where $y(m)$ is the function given in Lemma \ref{lemmaD'}.
Note that there exists a constant $c_3(d)$ such that for all
 $x\in\mathbb Z^d$
 contained in $C(\alpha,l,R(2^me_1))$ and such that  $|R(2^{m}e_1)-x|_1\le 1$
one has that
there exists an admissible path with at most $c_3 2^m$ lattice points joining $0$ and $x$.
We denote this path by

$$
(0,y_1,\ldots,y_n=x)
$$
noting that $n\le c_3 2^m$.

\medskip
The general idea to finish the proof is to push forward the walk up to site $x$
with the help of uniform ellipticity in direction $l$ and then we make use of Lemma (\ref{lemmaD'}) to ensure that the walk remains inside the cone.

Therefore,  by (\ref{N}) and Lemma  (\ref{lemmaD'}) we can conclude that
for all $2\le i \le d$ one has that
\begin{equation}
\label{L1}
P_x [D_{l_{i+}}=\infty] \geq y(m),
\end{equation}
along with
\begin{equation}
\label{L2}
P_x [D_{l_{i-}}=\infty] \geq y(m).
\end{equation}
Define the event that the random walk starting from $0$ following that path
$(0,y_1,\ldots, y_n)$ as

$$
A_n:=\{(X_0,\ldots, X_n)=(0,y_1,\ldots, y_n)\}.
$$
Now notice that
$$
P_0[D'=\infty] \geq
$$
\begin{equation}
\label{ineD}
\begin{split}
P_0\left[A_n,
(D_{l_{i-}}=\infty)\circ \theta_{n}, (\ D_{l_{i+}}=\infty)\circ \theta_{n}\ {\rm for}\
2\le i\le d\right].
\end{split}
\end{equation}
On the other hand, by definition of the annealed law, together with the strong Markov property
we have that
\begin{eqnarray}
\nonumber
&  P_0[A_n, (D_{l_{i-}}=\infty)\circ \theta_{n}, (\ D_{l_{i+}}=\infty)\circ \theta_{n}
\ {\rm for}\ 2\le i\le d]=\\
\label{din}
&  \mathbb{E}\left[P_{0,\omega}[A_n],
P_{x,\omega}[D_{l_{i-}}=\infty , D_{l_{i+}}=\infty\ {\rm for}\ 2\le i\le d]\right].
\end{eqnarray}
Using the uniform ellipticity assumption $(UE)|l$, along with (\ref{L1}) and (\ref{L2}), we can see that (\ref{din}) is bounded from
below by

\begin{eqnarray}
(2\kappa)^{c_3 2^{m}}\left(1-2 (d-1) (1- y(m))\right).
\end{eqnarray}
By virtue of our choice of $m$ in (\ref{ym}), we see that there exists a constant $c_2$ just depending on
the dimension (we recall that $m$ is fixed at this point of the proof), such that

\begin{equation}
\label{const}
c_4:=(2\kappa)^{c_3 2^{m}} \left(1-2(d-1) (1- y(m)) \right)>0
\end{equation}
Finally, in view of the inequalities (\ref{ineD}) and (\ref{din}) it follows that
$$
P_0[D'=\infty]\geq c_4.
$$
\end{proof}

\section{Polynomial control of regeneration positions}
\label{section5}

In this section, we define
an approximate regeneration times as done in \cite{CZ01}, which
will depend on a distance parameter $L>0$.
We will then show that these times, assuming $(PC)_{M,c}|l$
for $M$ large enough, and cone-mixing, when scaled by $\kappa^L$,
define approximate regeneration positions with a finite
second moment.

\subsection{Preliminaries}
We recall the definition of approximate renewal time given in \cite{CZ01}.
Let $W:=\mathcal{E}\cup\{0\}$
[c.f. (\ref{espacio-epsilon})]
 and  endow the space $W^{\mathbb{N}}$ with the canonical $\sigma-$algebra $\mathcal W$ generated
 by the
 cylinder sets.
 For fixed  $\omega\in \Omega$ and $\varepsilon
=(\varepsilon_0,\varepsilon_1,\ldots)\in W^{\mathbb{N}}$, we denote by
$P_{\omega, \varepsilon}$ the law of the Markov chain $\{X_n\}$ on $(\mathbb{Z}^d)^{\mathbb{N}}$, so that $X_0=0$ and with transition probabilities defined for $z\in \mathbb{Z}^d$, $e, |e|=1$ as
$$
P_{\omega,\varepsilon}[X_{n+1}=z+e|X_n=z]=\mathds{1}_{\{\varepsilon_{n}=e\}}+\frac{\mathds{1}_{\{\varepsilon_{n}=0\}}}{1-\kappa|\mathcal E |}[\omega(z,e)-\kappa\mathds{1}_{\{e\in \mathcal{E}\}}].
$$
Call $E_{\omega,\epsilon}$ the corresponding expectation.
Define also the product measure $Q$, which to each sequence of the form $\varepsilon
\in W^{\mathbb{N}}$ assigns the probability
$Q(\varepsilon_1=e):=\kappa$, if $e\in \mathcal{E}$, while $Q(\varepsilon_1=0)=1-\kappa|\mathcal{E}|$, and denote by $E_Q$ the corresponding expectation.

Now let $\mathfrak G$ be the $\sigma$-algebra on $(\mathbb Z^d)^{\mathbb N}$ generated by
cylinder sets, while $\mathfrak F$ be the $\sigma$-algebra on $\Omega$ generated by
cylinder sets. Then, we can define for fixed $\omega$ the measure

$$
\overline{P}_{0,\omega}:= Q \otimes P_{\omega,\varepsilon}
$$
on the space
$(W^{\mathbb{N}} \times (\mathbb{Z}^{d})^{\mathbb{N}},  \mathcal{W} \times \mathfrak{G})$,
and also

$$
\overline{P}_0:=\mathbb{P}\otimes Q \otimes P_{\omega,\varepsilon}
$$
on
$(\Omega \times W^{\mathbb{N}} \times (\mathbb{Z}^{d})^{\mathbb{N}}, \mathfrak{F}\times \mathcal{W} \times \mathfrak{G})$,
denoting by $\bar E_{0,\omega}$ and $\bar E_0$ the corresponding expectations.
A straightforward computation makes us conclude that the law of $\{X_n\}$ under $\bar P_{0,\omega}$ coincides with its law under $P_{0,\omega}$ and that its law under $\overline{P}_0$ coincides with its law under $P_0$.

\medskip

Let $q$ be a positive real number such that for all $1\le i\le d$,

$$
u_i:=l_i q
$$
 is an integer.
Define now the vector
$u:=(u_1,\ldots, u_d)$.
From now on, we fix a particular sequence $\overline{\varepsilon}$ in $\mathcal{E}$ of length $p:=|u|_1$ whose
components sum up to $u$:

$$
\overline{\varepsilon}:=(\overline{\varepsilon}_1,\ldots,\overline{\varepsilon}_{p}),
$$
together with
\begin{eqnarray*}
&\overline{\varepsilon}_1=\overline{\varepsilon}_2=\ldots=\overline{\varepsilon}_{|u_1|}:= \ \mbox{sgn}(u_1)e_1,\\
&\overline{\varepsilon}_{|u_1|+1}=\overline{\varepsilon}_{|u_1|+2}=\ldots=\overline{\varepsilon}_{|u_1|+|u_2|}:=\ \mbox{sgn}(u_2)e_2\\
&\vdots \\
&\overline{\varepsilon}_{p-|u_d|+1}=\ldots=\overline{\varepsilon}_{p}:=\ \mbox{sgn}(u_d)e_d.
\end{eqnarray*}
Without loss of generality we can assume  that $l_1 \not=0$.  And by taking $\alpha$ small enough that
$$
\overline{\varepsilon}_1,\overline{\varepsilon}_1+\overline{\varepsilon}_2,\ldots\overline{\varepsilon}_{p}
$$
are inside of $\mathcal{C}(0,l,\alpha)$.
For $L\in p\mathbb N$ consider the sequence $\bar\varepsilon^{(L)}$ of length $L$,
defined as the concatenation $L/p$ times with itself of the sequence $\bar\varepsilon$,
so that

$$
\overline{\varepsilon}^{(L)}=(\bar\varepsilon_1,\ldots,\bar\varepsilon_p,\ldots,
\bar\varepsilon_1,\ldots,\bar\varepsilon_p).
$$

\noindent Consider the filtration $\mathcal G:=\{\mathcal G_n:n\ge 0\}$ where

$$
\mathcal{G}_n:=\sigma((\varepsilon_i,X_i),i\leq n).
$$
Define $S_0:=0$,

\begin{eqnarray*}
&S_1:=
\inf\{n \geq L: \ X_{n-L}\cdot l>\max\{X_m \cdot l:m<n-L\},\\
& (\varepsilon_{n-L},\ldots,\varepsilon_{n-1})=\overline{\varepsilon}^{(L)}\}
\end{eqnarray*}
together with
$$
R_1:=D'\circ\theta_{S_1}+S_1.
$$
We can now recursively define for $k\ge 1$,

\begin{eqnarray*}
& S_{k+1}:=
\inf\{n \geq R_k : \ X_{n-L}\cdot l>\max\{X_m \cdot l:m<n-L\},\\
& (\varepsilon_{n-L},\ldots,\varepsilon_{n-1}=\overline{\varepsilon}^{(L)}))\}
\end{eqnarray*}
and
$$
R_{k+1}:=D'\circ\theta_{{S}_{k+1}}+S_{k+1}.
$$
Clearly,
$$
0=S_0\leq S_1\leq R_1\leq\cdots\infty,
$$
the inequalities are strict if the left member of the corresponding inequality is finite, and
the sequences $\{S_k:k\ge 0\}$ and $\{R_k:k\ge 0\}$ are $\mathcal G$-stopping times. On the other hand, we can check that $\bar P_0-$a.s. one has that $S_1<\infty$ along with the fact $\bar P_0-$a.s. on the set
\begin{align}
\label{remark}
\{\lim X_n\cdot l=\infty\}\cap\{R_{k}<\infty\}\ \mbox{one has too that}&\\
\nonumber
S_{k+1}<\infty.
\nonumber
\end{align}
Put
$$
K:=\inf\{k\geq1:{S}_k<\infty, {R}_{k}=\infty\}
$$
and define the \textit{approximate regeneration time}

\begin{equation}
\label{taudef}
\tau^{(L)}:={S}_{K}.
\end{equation}
We see that the random variable $\tau^{(L)}$ is the first time $n$ in which the walk has reached a record at time $n-L$ in direction $l$, and then  the walk goes on $L$ steps in the direction $l$ by means of the action of $\overline{\varepsilon}^{(L)}$ and finally after this time $n$, never exits the cone $C(X_n,l,\alpha)$.

 The following lemma is required to show that the approximate renewal times
are $\bar P_0$-a.s. finite. Its can be proved using a slight
variation of the argument given in page 517 of Sznitman \cite{Sz03}.

\medskip

\begin{lemma}
\label{transience1}
 Consider a random walk in a random environment. Let $l\in\mathbb{S^*}^{d-1}$,
$M\ge d+1$ and $c>0$ and
assume that $(PC)_{M,c}|l$ is satisfied. Then the random walk is transient in direction $l$.
\end{lemma}
\begin{proof}  The proof can be obtained following for example the
argument presented in page 517 of \cite{Sz03}, through the use
of Borel-Cantelli and the fact that for any $M>0$ we have that
$$
P_0[\limsup_{n\rightarrow\infty}X_n\cdot l=\infty]=1.
$$
\end{proof}
\medskip

We can now prove the following stronger version of Lemma 2.2 of \cite{CZ01}.

\begin{lemma}
Assume $(CM)_{\alpha,\phi}|l$, $(UE)|l$ and $(PC)_{M,c}$ for $M>6d-3, c>0$. Then there exists a positive $L_0\in |u|_1\mathbb{N}$,  such that

$$
\phi(L_0)+P_0[D'<\infty]<1,
$$
and $\tau^{(L)}<\infty$, $P_0$-a.s. are fulfilled for each  $L \geq L_0, \ L\in|u|_1 \mathbb{N}$.
\end{lemma}
\begin{proof} Following the arguments in the proof of Lemma 2.2. of \cite{CZ01} (using $u$ instead of $l$), one has that:
\begin{equation}
\bar P_0[R_k<\infty]\leq (\phi(L_0)+P_0[D'<\infty])^k
\end{equation}
From the assumption $(CM)_{\alpha,\phi}|l$, we have $\phi(L)\rightarrow 0$ as $L\rightarrow\infty$. On the other hand, by Lemma
\ref{D'},
$$
P_0[D'<\infty]<1.
$$
Therefore, we can find a $L_0$ with the property:
$$
\phi(L)+P_0[D'<\infty]<1,
$$
for all $L\geq L_0, \  L\in \mathbb{N}|l|_1$.\\
Then, via  Borel-Cantelli Lemma, one has that
$\bar P_0-$ almost surely
\begin{equation}
\inf\{n\geq1 : \ R_n=\infty\}<\infty,
\end{equation}
holds. Now, observe that $\bar P_0-$ almost surely:
\begin{equation}
\inf\{n\geq1 : \ R_n=\infty\}=\inf\{n\geq1 : R_{n-1}<\infty \ R_n=\infty\}
\end{equation}
In turn, using (\ref{remark}) which is satisfied in view of Lemma \ref{transience1}, turns out that
$$
\inf\{n\geq1 : S_n<\infty \ R_n=\infty\}=K<\infty
$$
$\bar P_0-$ almost surely.
\end{proof}

\medskip

\noindent Finally, we can state the following  proposition, which gives a control
on the second moment of the position of the random walk at the first regeneration position.
Define for $x\in \mathbb{Z}^d$ and $L>0$ the $\sigma-$algebra
$$
\mathfrak{F}_{x,L}:=\sigma\left\{\omega(y,\cdot); \ y\cdot u\leq x\cdot u-\frac{L}{|u|_1}|u|_2\right\}.
$$
\medskip

\begin{proposition}
\label{prop-sec5} Fix $l\in\mathbb{S^*}^{d-1}$, $\alpha>0$, $M>0$
and $\phi:[0,\infty)\to [0,\infty)$ such that $\lim_{r\to\infty}\phi(r)=0$.
Assume that $0<\alpha<\min\{\frac{1}{9},\frac{1}{2c+1}\}$
and that $(CM)_{\alpha,\phi}|l$, $(UE)|l$ and $(PC)_{M,c}|_l$ hold.
Then, there exists a constant $c_{5}$, such that

\begin{equation}
\label{eq-prop5}
\bar E_0[(\kappa^L X_{\tau^{(L)}}\cdot l)^2 | \mathfrak{F}_{0,L}]\leq c_{5}.
\end{equation}

\end{proposition}

\medskip

\subsection{Preparatory results}
Now we are in position to prove the main proposition of this section. Before we do this,
 we will prove a couple of lemmas.

\medskip

\begin{lemma}
\label{lemadinf}
Assume that $(CM)_{\alpha,\phi}|l$ holds. Then, for each $x \in \mathbb{Z}^d$ one has that
$$\left|\mathbb{E}[P_{x,\omega}[D'=\infty]| \mathfrak{F}_{x,L}]- P_0[D'=\infty]\right|\leq \phi(L)$$
holds a.s.
\end{lemma}

\begin{proof}
For each
$A\in \mathfrak{F}_{x,L}$, we  define

\begin{equation}
\label{nua}
\nu[A]:=\mathbb{E}[P_{x,\omega}[D'=\infty]\mathds{1}_{A}]
\end{equation}
and

\begin{equation}
\label{mua}
\mu[A]:= \left(P_0[D'=\infty]+\phi(L)\right)\mathbb{P}[A]-\nu[A].
\end{equation}
Clearly (\ref{nua}) defines a measure on $(\Omega,\mathfrak F_{x,L})$.
We will show that (\ref{mua}) also.
Indeed, take an $A\in \mathfrak{F}_{x,L}$  and note that $P_{x,\omega}[D=\infty]$ is
$\sigma\{ \omega(y,\cdot), \, y\in C(x, l , \alpha)\} $-measurable. Therefore, by
assumption $(CM)_{\alpha,\phi}|l$ one has that

$$
\nu[A]\leq P_{0}[D'=\infty]\mathbb{P}[A]+\phi(L)\mathbb{P}[A].
$$
Consequently, (\ref{mua}) defines a measure $\mu$ on $(\Omega,\mathfrak{F}_{x,L})$.
Consider the  increasing sequence $\{A_n: n\geq1\}$ of $ \mathfrak{F}_{x,L}$-measurable sets  defined by

$$
A_n:=\left\{\omega \in \Omega:
 \mathbb{E}[P_{x,\omega}[D'=\infty]| \mathfrak{F}_{x,L}]> P_{0}[D'=\infty]+\phi(L)+\frac{1}{n}\right\}
$$
and define
$$
A:=\bigcup_{n\geq1}A_n.
$$
Observe that for each $n\geq 1$ we have that

\begin{eqnarray*}
&0\leq \mu(A_n)= (P_0[D=\infty]+\phi(L))
\mathbb P[A_n]-\mathbb{E}[\mathbb{E}[P_{x,\omega}[D'=\infty]| \mathfrak{F}_{x,L}]\mathds{1}_{A_n}]\\
&\leq -\frac{1}{n}\mathbb{P}[A_n].
\end{eqnarray*}
Therefore, one has that for each $n\geq1$, $\mathbb{P}[A_n]=0$
and consequently $\mathbb{P}[A]=0$.
Observing that
$$
A=\{\omega\in \Omega: \mathbb{E}[P_{x,\omega}[D'=\infty]|\mathfrak{F}_{x,L}]> P_{0}[D'=\infty]+\phi(L) \},
$$
we see that

\begin{equation}
\label{show}
\mathbb{E}[P_{x,\omega}[D'=\infty]| \mathfrak{F}_{x,L}]- P_0[D'=\infty]\leq \phi(L).
\end{equation}
One can prove that

$$-\phi(L)\le \mathbb{E}[P_{x,\omega}[D'=\infty]| \mathfrak{F}_{x,L}]- P_0[D'=\infty]$$
following the same argument used to show (\ref{show}), but changing the
event $\{D'=\infty\}$ by $\{D'<\infty\}$.

\end{proof}

\medskip

\noindent The second lemma that will be needed to prove Proposition \ref{prop-sec5} is the following one.
To state it define

$$
\mathfrak{M}:=\sup_{0\leq n \leq D'}(X_n-X_0)\cdot u,
$$
\begin{equation}
\nonumber
D'(0):=\inf\{n\geq0: \  X_n\not\in C(0,l,\alpha)\},
\end{equation}
and for $a\in \mathbb{R}$
\begin{eqnarray}
\nonumber
&T^l_a:=\inf\{n\geq0: \  X_n\cdot l\geq a\}\qquad {\rm and}\\
\label{tmayor}
&\bar T^l_a:=\inf\{n\geq0: \  X_n\cdot l> a\}.
\end{eqnarray}

\medskip

\begin{lemma}
\label{M}
Let $M>4d+1$
and

\begin{equation}
\label{restriction}
2c+1\leq\frac{1}{\alpha}.
\end{equation}
Assume that $(PC)_{M,c}|l$ is satisfied.
Then, there exists $c_6=c_6(d)>0$ such that a.s.
one has that

$$
E_0[\mathfrak{M}^2, D'< \infty |  \mathfrak{F}_{0,L}]\leq c_6.
$$
$\mathbb{P}-$ almost surely.
\end{lemma}

\begin{proof}
To simplify the proof, we will show that the second moment of

$$
\mathfrak{M'}:=\sup_{0\leq n \leq D'}(X_n-X_0)\cdot l$$
is bounded from above.
Note that

\begin{eqnarray}
\nonumber
&E_0[\mathfrak{M'}^2, D'< \infty |  \mathfrak{F}_{0,L}]\leq
P_0[D'< \infty \mid  \mathfrak{F}_{0,L}]\\
\label{M2} &
+\sum_{m\geq0} 2^{2(m+1)}P_0[2^m\leq \mathfrak{M'}<2^{m+1},D'< \infty \mid \mathfrak{F}_{0,L}].
\end{eqnarray}
Therefore, it is enough to obtain an appropriate upper bound
 of the probability when $m$ is large
$$
P_0[2^m\leq \mathfrak{M'}<2^{m+1},D'< \infty \mid \mathfrak{F}_{0,L}].
$$
Note that,

\begin{eqnarray}
\nonumber
&P_0[2^m\leq \mathfrak{M'}<2^{m+1},D'< \infty \mid \mathfrak{F}_{0,L}]\\
\nonumber&\leq
P_0[T^l _{2^m}<D'<\infty, T^l _{2^{m+1}}\circ\theta_{T_{2^m}}
> D'(0)\circ\theta_{T_{2^m}}
 \mid \mathfrak{F}_{0,L}]\\
\nonumber
&\leq
\nonumber
P_0[ X_{T^l _{2^m }}\not \in \partial^+B_{2^m,c2^m,l}(0),  T^l _{2^m }<D'<\infty  \mid \mathfrak{F}_{0,L}]\\
\label{Pm}
&+P_0[X_{T^l _{2^m}}\in \partial^+B_{2^m,c2^m,l}(0),
T^l _{2^{m+1}}\circ\theta_{T_{2^m}}
> D'(0)\circ\theta_{T_{2^m}}\mid \mathfrak{F}_{0,L}].
\end{eqnarray}
Using $(PC)_{M,c}|l$, we get the following upper bound for the first term
of the rightmost expression in  (\ref{Pm}),

\begin{eqnarray}
\nonumber
&P_0[X_{T_{B_{2^m,c2^m,l}(0)}}\not\in \partial^+B_{2^m,2^m,l}(0),(X_n)_{0\leq n \leq T_{B_i(0)}}\subset (H_{0,l})^c|\mathcal H_{0,l}]\\
\label{b1M}
&\leq
2^{-M m}.
\end{eqnarray}
As for the second term in the rightmost expression
in (\ref{Pm}), it will be useful to introduce the set
$$
F_{m}:=\partial^+B_{2^m,c2^m,l}(0).
$$
Now, by the strong Markov property we have the bound

\begin{eqnarray}
\nonumber
& P_0[X_{T_{B_{2^m,c2^m,l}}(0)}\in \partial^+B_{2^m,2^m,l}(0),  T^l _{2^{m+1}}\circ\theta_{T^l _{2^m}}> D'(0)\circ\theta_{T^l _{2^m}}\mid \mathfrak{F}_{0,L}]\\
\label{Pyc}
&\leq \sum_{y\in F_{m}}P_y[T^l _{2^{m+1} }>D'(0) \mid \mathfrak{F}_{0,L}].
\end{eqnarray}
In order to estimate this last conditional probability, we obtain a lower bound for its complement as follows.
To simplify the computations which follow, for each $x\in\mathbb Z^d$ we
introduce the notation

$$
B_x:=B_{2^{m-1},c2^{m-1},l}(x).
$$
Now, note that under the assumption (\ref{restriction})
we have that

$$
c \left(2^m+2^{m-1}\right)\leq \cot(\beta)2^{m-1},
$$
which implies that
the boxes $B_y$  and $B_z$,
 for all
$y\in F_{m}$ and $z\in \partial^+ B_y$,
are inside the cone $C(0,l,\alpha)$ (see  Figure \ref{fig}).
\begin{figure}[h]
  \begin{center}
  \includegraphics[width=8cm]{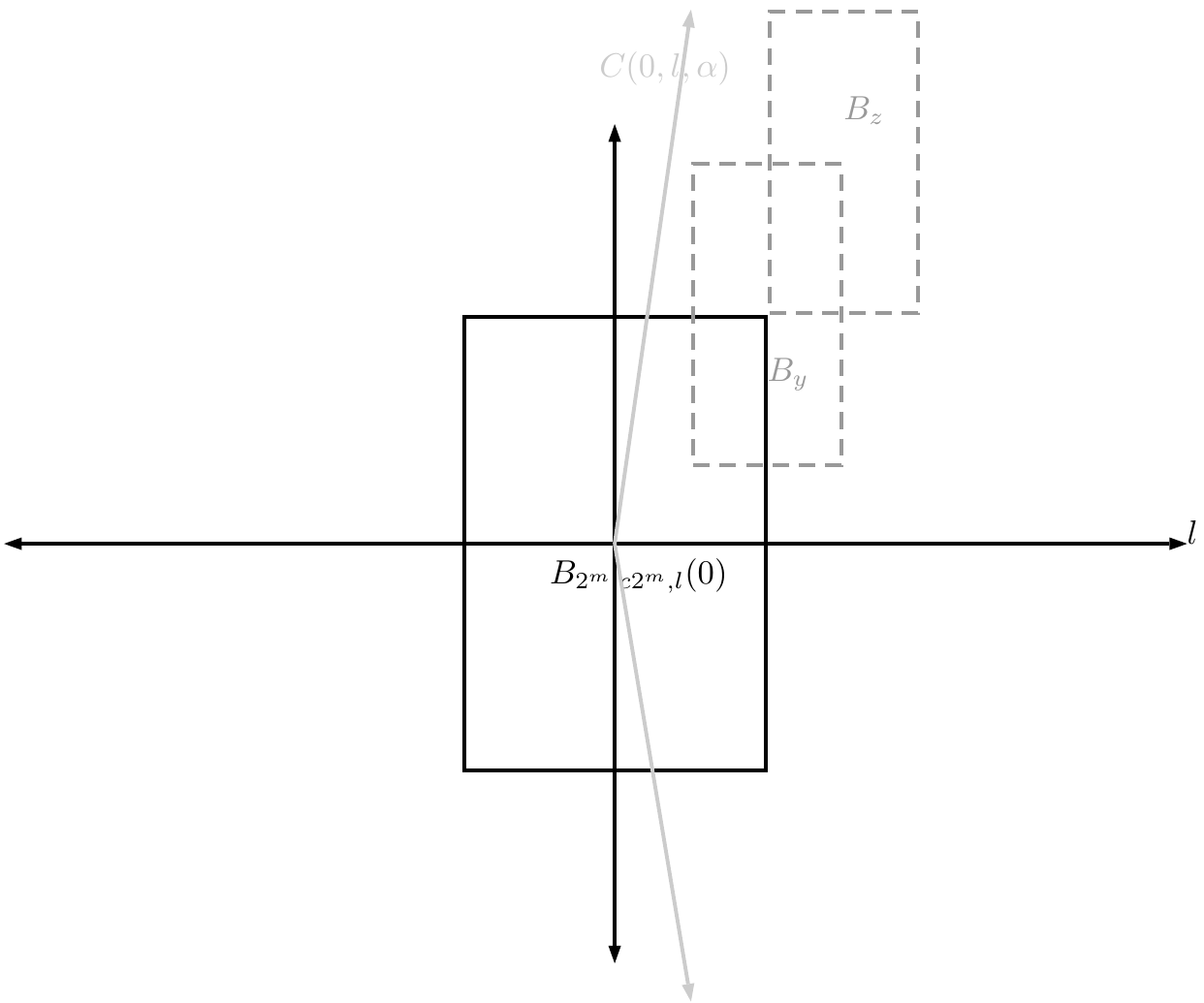}\\
  \caption{The boxes $B_y$ and $B_z$ are inside of $C(0,l,\alpha)$.}\label{fig}
  \end{center}
\end{figure}

\medskip
\noindent
Therefore, fixing $y \in F_m$,  it follows
that
\begin{eqnarray}
\nonumber
& P_y[T^l _{2^{m+1}} < D'(0) \mid \mathfrak{F}_{0,L}]\geq\\
\nonumber
&
 \sum_{z \in  \partial^+ B_y}\mathbb{E}[P_{y,\omega}[X_{T_{B_y}}\in
\partial^+ B_y,\\
\label{Pyc2}
& X_{T_{B_y}}=z,
(X_{T_{B_z}}\in\partial^+ B_z)\circ\theta_{T_{B_y}}]| \mathfrak{F}_{0,L}].
\end{eqnarray}
To estimate the right-hand side of the above
inequality, it will be convenient to introduce the set
$$
\bar F_{m}:=\partial[\cup_{y\in F_{m}}B_y]\cap \{ R([ 2^{m-1}+2^m,\infty)\times \mathbb{R}^{d-1})\},
$$
and the event

\begin{eqnarray*}
&G_{\bar F_{m}}:=\{\omega\in \Omega : \ P_{z,\omega}[X_{T_{B_z}}\in \partial^+B_z]>\\
&1-2^{-\frac{M(m-1)}{2}}, \ \mbox{for all }z \in \bar F_{m} \}.
\end{eqnarray*}
Using the strong Markov property, we can now bound from below the right-hand side
of  inequality (\ref{Pyc2}) by
\begin{equation}
\label{expression}
(1-2^{-\frac{M(m-1)}{2}})\left(P_y[X_{T_{B_y}}\in \partial^+ B_y| \mathfrak{F}_{0,L}]-
P_y[(G_{\bar F_{m}})^c|\mathfrak{F}_{0,L}]\right).
\end{equation}
In turn, by means of the polynomial condition and  the
fact that  the boxes $B_y$ and $B_z$ are inside the cone $C(0,l,\alpha)$
 we  see that (\ref{expression})  is greater than or equal to

\begin{equation}
\label{in1}
(1-2^{-\frac{M(m-1)}{2}})\left(1-2^{-M(m-1)}-P_y[(G_{\bar F_{m}})^c|\mathfrak{F}_{0,L}]\right).
\end{equation}
Now, note that

\begin{eqnarray}
\nonumber
&P_y[(G_{\bar F_{m}})^c|\mathfrak{F}_{0,L}]\leq  \sum_{x\in \bar F_{m} }2^{\frac{M(m-1)}{2}}P_{x}[
X_{T_{B_x}}\not\in \partial^+B_x|\mathfrak{F}_{0,L}]\\
\label{expression2}
&\le
|\bar F_{m}|2^{-\frac{M(m-1)}{2}}\le (4c)^{d-1} 2^{m(d-1)}
2^{-\frac{M(m-1)}{2}}.
\end{eqnarray}
where in the first inequality we have used  Chebyshev  inequality,
in the second one the
assumption that $(PC)_{M,c}|_l$ is satisfied
and in the third one the bound $|\bar F_{2m}|\le (4c)^{d-1}
2^{m(d-1)}$.

Consequently inserting the estimates (\ref{expression2}) into
(\ref{in1}) and combining this with inequality (\ref{Pyc2})
we conclude that

\begin{eqnarray}
\nonumber
&P_y[T^l _{2^{m+1}} \leq D'(0) \mid \mathfrak{F}_{0,L}]\geq \ (1-2^{-\frac{M(m-1)}{2}})\times\\
\nonumber
&(1-2^{-\frac{M(m-1)}{2}}-(4c)^{d-1}2^{m(d-1)}2^{-\frac{M(m-1)}{2}})\\
&
\label{Py}
\ge 1-3(4c)^{d-1}2^{m(d-1)}2^{-\frac{M(m-1)}{2}}.
\end{eqnarray}
Using the bound (\ref{Py}) in (\ref{Pyc}),
together with the estimate $|F_{m}|\leq (2c)^{d-1}2^{m(d-1)}$, we
 see that

\begin{eqnarray}
\nonumber
& P_0[X_{T_{B_{2^m,c2^m,l}}(0)}\in \partial^+B_{2^m,2^m,l}(0),  T^l _{2^{m+1}}\circ\theta_{T^l _{2^m}}> D'(0)\circ\theta_{T^l _{2^m}}\mid \mathfrak{F}_{0,L}]\\
\label{cestim}
&\leq
3(4c)^{2(d-1)}2^{2m(d-1)}2^{-\frac{M(m-1)}{2}}.
\end{eqnarray}
Combining the estimates (\ref{cestim}), (\ref{b1M}), (\ref{Pm}) with (\ref{M2}) we conclude that

\begin{eqnarray}
\nonumber
&E_0[\mathfrak{M'}^2, D'< \infty |  \mathfrak{F}_{0,L}]\\
\nonumber
&\le1+4(4c)^{2(d-1)}\sum_{m\geq0} 2^{2(m+1)}2^{2m(d-1)}2^{-\frac{M(m-1)}{2}}\\
\nonumber
&\le
1+4(4c)^{2(d-1)}\sum_{m\geq0} 2^{-m}\le c_6,
\end{eqnarray}
where in the second to last inequality we have used the fact that $M>4d+1$
and $c_6$ is a constant that does not depend on $L$.
This completes the proof of the Lemma.
\end{proof}

\medskip
\subsection{Proof of Proposition \ref{prop-sec5}}
 To simplify the
computations, we introduce the notation

\begin{eqnarray}
\nonumber
&b=b(L):=P_0(D'<\infty)+\phi(L),\\
\nonumber
&b'=b'(L):=P_0(D'=\infty)+\phi(L)
\end{eqnarray}
and  $E_{\mathbb P\otimes Q}:=\mathbb E E_Q$.
Furthermore, it will be necessary to define for each $j\ge 0$ and $n\ge L+j$ the events
$$
D_{j,n}:=\{\varepsilon\in W^{\mathbb N}:(\varepsilon_{m},\ldots,\varepsilon_{m+L-1})\ne\overline{\varepsilon}^{(L)}\ {\rm for}\ {\rm all}\ j\le m\le j+n-L+1\}.
$$
The following lemma, whose proof is presented
in Appendix \ref{appendix}, will be useful
in the proof of Proposition \ref{prop-sec5}.

\medskip

\begin{lemma}
\label{qapp}
There exists a constant $c_{7}$ such that
for all $n\ge L^2$ one has that

$$
Q[D_{0,n}]\le
(1-c_{7}L^2\kappa^L)^{\left[\frac{n}{L^2}\right]}.
$$
\end{lemma}

\medskip

\noindent
We now present the proof
of Proposition \ref{prop-sec5}, divided in several steps.
For the sake of simplicity, we will write $\tau$ instead of $\tau^{(L)}$.

\medskip

\noindent\emph{Step 0}.
We first note that

\begin{eqnarray}
\nonumber
&\bar E_0[(X_{\tau}\cdot u)^2 \mid \mathfrak{F}_{0,L}]=\\
\nonumber
&
\displaystyle{\sum_{k=1}^\infty}
\displaystyle{\sum_{k'=0}^{k-1}}\bar E_0[(X_{S_{k'+1}}\cdot u)^2-(X_{S_{k'}}\cdot u)^2,S_k<\infty, D'\circ\theta_{S_k}=\infty \mid \mathfrak{F}_{0,L}].\\
\label{decomXt}
\end{eqnarray}
Throughout the subsequent steps of the proof we will estimate
the right-hand side of (\ref{decomXt}).

\medskip

\noindent\emph{Step 1}. Here we will prove the following estimate
valid for all $k\ge 1$ and $0\le k'<k$.

\begin{eqnarray}
\nonumber
&\bar E_0[(X_{S_{k'+1}}\cdot u)^2-(X_{S_{k'}}\cdot u)^2, S_k<\infty,
 D'\circ\theta_{S_k}=\infty\mid \mathfrak{F}_{0,L}]\\
\label{number7.1}
&\le b' b^{k-k'-1}
 \bar E_0[(X_{S_{k'+1}}\cdot u)^2-(X_{S_{k'}}\cdot u)^2, S_{k'+1}<\infty
 \mid \mathfrak{F}_{0,L}].
\end{eqnarray}
Furthermore, define the set

$$
H^{L}:=\left\{y\in \mathbb{Z}^d: y\cdot u \geq L\frac{|u|_2}{|u|_1}\right\}.
$$
Then, for each $0\le k'<k$, one has that
\begin{eqnarray}
\nonumber
& \bar E_0[(X_{S_{k'+1}}\cdot u)^2-(X_{S_{k'}}\cdot u)^2, S_k<\infty, D'
\circ\theta_{S_k}=\infty \mid \mathfrak{F}_{0,L}]\\
\nonumber
&\displaystyle{=\sum_{n\geq1, x\in H^L}}E_{\mathbb{P}\otimes Q}[E_{\omega,\varepsilon}[(X_{S_{k'+1}}\cdot u)^2-(X_{S_{k'}}\cdot u)^2, S_k=n,\\
\nonumber
&X_{S_k}=x, D'\circ\theta_{n}=\infty \mid \mathfrak{F}_{0,L}]\\
\nonumber
&\displaystyle{=\sum_{n\geq1, x\in H^L}}E_{\mathbb{P}\otimes Q}[E_{\omega,\varepsilon}[(X_{S_{k'+1}}\cdot u)^2-(X_{S_{k'}}\cdot u)^2, S_k=n,X_{n}=x]\\
\nonumber
&P_{\vartheta_x \omega,\ \theta_n\varepsilon}[D'=\infty]\mid \mathfrak{F}_{0,L}]\\
\nonumber
&\displaystyle{=\sum_{x\in H^L}}\mathbb{E}[\bar E_{0,\omega}[(X_{S_{k'+1}}\cdot u)^2-(X_{S_{k'}}\cdot u)^2, S_k<\infty,X_{S_k}=x]\\
\nonumber
&P_{x,\omega}[D'=\infty]\mid \mathfrak{F}_{0,L}], \label{number1}\\
\end{eqnarray}
where here for each $x\in\mathbb Z^d$, $\vartheta_x$ denotes the canonical
space shift in $\Omega$ so that $\vartheta_x\omega(y)=\omega(x+y)$, while
for each $n\ge 0$, $\theta_n$ denotes the canonical time shift in
the space $W$ so that $(\theta_n\epsilon)_m=\epsilon_{n+m}$,
in the first equality we have used the fact that the value
of $X_{S_k}\cdot u\ge X_{S_1}\cdot u$,
in the second equality the Markov property
and in the last equality we have used the independence of the
coordinates of $\epsilon$ and the fact that the law of the
random walk is the same under $P_{x,\omega}$ and under
$E_{Q}P_{\vartheta_x \omega,\theta_n\epsilon}$.

Moreover, by  the fact that the first factor inside the
expectation of the right-most expression of (\ref{number1}) is $\mathfrak F_{x,L}$-measurable,
the right-most expression in (\ref{number1}) is equal to
\begin{eqnarray}
\nonumber
& \displaystyle{\sum_{x\in H^{L}}}\mathbb{E}[\bar E_{0,\omega}[(X_{S_{k'+1}}\cdot u)^2-(X_{S_{k'}}\cdot u)^2, S_k<\infty,X_{S_k}=x]\\
\label{number2}
& \mathbb{E}[P_{x,\omega}[D'=\infty]\mid \mathfrak{F}_{x,L}]\mid \mathfrak{F}_{0,L}].
\end{eqnarray}
Applying next Lemma \ref{lemadinf} to (\ref{number2}), we see that
\begin{eqnarray}
\nonumber
&\sum_{x\in H^{L}}\mathbb{E}[\bar E_{0,\omega}[(X_{S_{k'+1}}\cdot u)^2-(X_{S_{k'}}\cdot u)^2, S_k<\infty,X_{S_k}=x]\\
\nonumber
&\times \mathbb{E}[P_{x,\omega}[D'=\infty]\mid \mathfrak{F}_{x,L}]\mid \mathfrak{F}_{0,L}]\\
\label{number3}
&\le b'
\bar E_0[(X_{S_{k'+1}}\cdot u)^2-(X_{S_{k'}}\cdot u)^2, S_k<\infty\mid \mathfrak{F}_{0,L}].
\end{eqnarray}
Next, observe that for $k'<k$ one has that

\begin{eqnarray}
\nonumber
&\bar E_0[(X_{S_{k'+1}}\cdot u)^2-(X_{S_{k'}}\cdot u)^2, S_k<\infty\mid \mathfrak{F}_{0,L}]\\
\nonumber
&=\bar E_0[(X_{S_{k'+1}}\cdot u)^2-(X_{S_{k'}}\cdot u)^2, R_{k-1}<\infty\mid \mathfrak{F}_{0,L}]\\
\nonumber
&=\sum_{x\in H^{L}}\mathbb{E}[\bar E_{0,\omega}[(X_{S_{k'+1}}\cdot u)^2-(X_{S_{k'}}\cdot u)^2, S_{k-1}<\infty, X_{S_{k-1}}=x,\\
\nonumber
&D' \circ \theta_{S_{k-1}}<\infty]\mid \mathfrak{F}_{0,L}]\\
\nonumber
&=\sum_{x\in H^{L}}\mathbb{E}[\bar E_{0,\omega}[(X_{S_{k'+1}}\cdot u)^2-(X_{S_{k'}}\cdot u)^2, S_{k-1}<\infty, X_{S_{k-1}}=x]\\
\nonumber
&P_{x,\omega}[ D' <\infty]\mid \mathfrak{F}_{0,L}]\\
\nonumber
&=\sum_{x\in H^{L}}\mathbb{E}[\bar E_{0,\omega}[(X_{S_{k'+1}}\cdot u)^2-(X_{S_{k'}}\cdot u)^2, S_{k-1}<\infty, X_{S_{k-1}}=x]\\
\label{number5}
&\mathbb{E}[P_{x,\omega}[ D' <\infty]\mid \mathfrak{F}_{x,L}]\mid \mathfrak{F}_{0,L}].
\end{eqnarray}
By  Lemma \ref{lemadinf}, we have that
$\mathbb{E}[P_{x,\omega}[ D' <\infty]\mid \mathfrak{F}_{x,L}]\leq b=P_0[D'<\infty] +\phi(L)$. Using this inequality to estimate  the last term in (\ref{number5}), we see
that

\begin{eqnarray}
\nonumber
&\bar E_0[(X_{S_{k'+1}}\cdot u)^2-(X_{S_{k'}}\cdot u)^2, S_k<\infty\mid \mathfrak{F}_{0,L}]\\
\nonumber
&\le b\bar E_0[(X_{S_{k'+1}}\cdot u)^2-(X_{S_{k'}}\cdot u)^2, S_{k-1}<\infty \mid \mathfrak{F}_{0,L}].
\end{eqnarray}
By induction on $k$ we get that

\begin{eqnarray}
\nonumber
&\bar E_0[(X_{S_{k'+1}}\cdot u)^2-(X_{S_{k'}}\cdot u)^2, S_k<\infty\mid \mathfrak{F}_{0,L}]\\
\label{number7}
&\le b^{k-k'-1}
 \bar E_0[(X_{S_{k'+1}}\cdot u)^2-(X_{S_{k'}}\cdot u)^2, S_{k'+1}<\infty \mid \mathfrak{F}_{0,L}].
\end{eqnarray}
Combining (\ref{number7}) with (\ref{number3}) we obtain (\ref{number7.1}).

\medskip

\noindent \emph{ Step 2}.
For $k\ge 1$ we define

\begin{equation}
\label{number10}
M_k:=\sup_{0\leq n \leq R_k}X_n \cdot u.
\end{equation}
Define also the sets parametrized by $k$ and $n\ge 0$

\begin{equation}
\label{ank}
A_{n,k}:=\left\{\varepsilon\in W^{\mathbb N} :\left(\varepsilon_{t^{(n)}_{k}},\varepsilon_{t^{(n)}_{k}+1},\ldots, \varepsilon_{t^{(n)}_{k}+L-1}\right)=\overline{\varepsilon}^{(L)}\right\}
\end{equation}
and

\begin{equation}
\label{bnk}
B_{n,k}:=\left\{\varepsilon\in W^{\mathbb N} : \left(\varepsilon_{t^{(j)}_{k}},\varepsilon_{t^{(j)}_{k}+1},\ldots,\varepsilon_{t^{(j)}_{k }+L-1}\right)\ne \overline{\varepsilon}^{(L)}\ {\rm for}\ {\rm all}\ 0\le j\le n-1\right\},
\end{equation}
where we define the  sequence of stopping times [c.f. (\ref{tmayor})]
parameterized by $k$ and  recursively on $n\ge 0$ by
\begin{equation}
\nonumber
t^{(0)}_k:=\bar T^{l}_{M_k}
\end{equation}
and the successive times where a record value of the projection of the
random walk on $l$ is achieved by

\begin{equation}
\nonumber
t_k^{(n+1)}:=\bar T^{l}_{X_{t_k^{(n)}}\cdot u}.
\end{equation}
In this step we will  show that for all $k\ge 0$
one has that

\begin{eqnarray}
\nonumber
&\bar E_0[(X_{S_{k+1}}\cdot u)^2-(X_{S_{k}}\cdot u)^2, S_{k+1}<\infty | \mathfrak{F}_{0,L}]\\
\nonumber
&\le
\sum_{n=0}^{L^2-1}\bar E_0[(X_{S_{k+1}}\cdot u)^2-(X_{S_{k}}\cdot u)^2,t^{(n)}_{k}<\infty,  A_{n,k}  \mid \mathfrak{F}_{0,L}]\\
\label{number11}
&\!\!+\sum_{n=L^2}^\infty\bar E_0[(X_{S_{k+1}}\cdot u)^2-(X_{S_{k}}\cdot u)^2,t^{(n)}_{k}<\infty, B_{n,k}, A_{n,k}  \mid \mathfrak{F}_{0,L}],
\end{eqnarray}
To prove (\ref{number11}), we have to introduce some further notations.
Now, note that
on the event $A_{n,k}\cap B_{n,k}$ one has that

$$
S_{k+1}=t^{(n)}_k+L.
$$
Thus, as a consequence of the definition of $S_{k+1}$, one has that
$\bar P_0$-a.s.

\begin{equation}
\label{number9}
\{S_{k+1}<\infty\}\subset \bigcup_{n\ge 0}\{t^{(n)}_{k}<\infty, B_{n,k}, A_{n,k}\}.
\end{equation}
Display (\ref{number11}) now follows directly from (\ref{number9}).

\medskip

\noindent\emph{Step 3}.
Here we will derive an upper bound for the
 two sums appearing in  the right-hand side  in
(\ref{number11}). In fact, we will prove that there is a constant
$c_8$ such that
for all $k\ge 1$ one has that

\begin{eqnarray}
\nonumber
&\sum_{n=0}^{L^2-1}\bar E_0[(X_{S_{k+1}}\cdot u)^2-(X_{S_{k}}\cdot u)^2,t^{(n)}_{k}<\infty,  A_{n,k}  \mid \mathfrak{F}_{0,L}]\\
\label{step3}
&\le c_8 \kappa^L\left(L^4 b^{k-1}
+L^2\bar E_0[X_{S_k}\cdot u,S_k<\infty|\mathfrak F_{0,L}]\right)
\end{eqnarray}
and

\begin{eqnarray}
\nonumber
&\sum_{n=L^2}^{\infty}\bar E_0[(X_{S_{k+1}}\cdot u)^2-(X_{S_{k}}\cdot u)^2,t^{(n)}_{k}<\infty,  B_{n,k},A_{n.k}  \mid \mathfrak{F}_{0,L}]\\
\nonumber
&\le
c_8\sum_{n=L^2}^\infty
 \kappa^L (1-c_{7}\kappa^L)^{\left[\frac{n}{L^2}\right]}
\left((n+L)^2 b^{k-1}\right.\\
\label{step3.1}
&\left. +(n+L)\bar E_0[X_{S_k}\cdot u,S_k<\infty|\mathfrak F_{0,L}]\right).
\end{eqnarray}
Note that for all $n\ge 0$ one has that

$$
X_{t^{(n+1)}_{k}}\cdot u \leq X_{t^{(n)}_{k}}\cdot u + |u|_\infty,
$$
and hence by induction on $n$ we get that

$$
X_{t^{(n)}_{k}}\cdot u\leq M_k+ (n+1) |u|_\infty.
$$
Therefore,  if we set
\begin{equation}
\label{L'}
L':=\frac{L|u|}{|u|_1}+ |u|_\infty\leq c_{9}L,
\end{equation}
where $c_{9}$ is a constant depending on $l$ and $d$,
we can see that $P_0$-a.s  on the event $\{t^{(n)}_{k}<\infty,\ A_{n,k} \}$
one has that

\begin{equation}
\label{number12}
X_{S_{k+1}}\cdot u\leq N_{k,n}:=M_k +n|u|_\infty +L'.
\end{equation}
Therefore, for all $0\le n\le L^2-1$ one has that

\begin{eqnarray}
\nonumber
&\bar E_0[(X_{S_{k+1}}\cdot u)^2-(X_{S_{k}}\cdot u)^2,t^{(n)}_{k}<\infty,  A_{n,k}  \mid \mathfrak{F}_{0,L}]\\
\nonumber
&
\le\bar E_0[N_{k,n}^2-(X_{S_{k}}\cdot u)^2,t^{(n)}_{k}<\infty,  A_{n,k}\mid \mathfrak{F}_{0,L}]\\
\nonumber
&=
\sum_{j= 0}^\infty \sum_{x\in\mathbb{Z}^d} E_{\mathbb{P} \otimes Q}[E_{\omega, \varepsilon}[N_{k,n}^2-(X_{S_k}\cdot u)^2,\\
\nonumber
&t^{(n)}_{k}=j,X_{j}=x]\mathds{1}_{\{(\varepsilon_{j},\ldots,\varepsilon_{j+L-1})=\overline{\varepsilon}^{(L)}\} } \mid \mathfrak{F}_{0,L}]\\
&
\label{number13}
\le \kappa^L \bar E_0[N_{k,n}^2-(X_{S_k}\cdot u)^2,R_{k }<\infty \mid \mathfrak{F}_{0,L}],
\end{eqnarray}
where in the  equality we have applied the Markov property
and in the second inequality the fact that $Q$ is a product measure
and that $R_k\le t^{(n)}_k$. Similarly  for all $n\ge L^2$
one has that

\begin{eqnarray}
\nonumber
&\bar E_0[(X_{S_{k+1}}\cdot u)^2-(X_{S_{k}}\cdot u)^2,t^{(n)}_{k}<\infty, B_{n,k}, A_{n,k}  \mid \mathfrak{F}_{0,L}]\\
\nonumber
&
\le\bar E_0[N_{k,n}^2-(X_{S_{k}}\cdot u)^2,t^{(n)}_{k}<\infty,  B_{n,k},A_{n,k}\mid \mathfrak{F}_{0,L}]\\
\nonumber
&\le \sum_{j= 0}^\infty\sum_{j'= j+n}^\infty\sum_{y\in\mathbb Z^d}E_{\mathbb{P}\otimes Q }\ [E_{\omega,\varepsilon}[N_{k,n}^2-(X_{S_{k}}\cdot u)^2,\\
\nonumber
&X_{t^{(0)}_{k}}=y,t^{(0)}_{k}=j ] P_{\theta_y \omega,\theta_{j} \varepsilon}[
D_{j,n}, t^{(n)}_{k}=j' ]
\mathds{1}_{\{ (\varepsilon_{j'},\ldots, \varepsilon_{j'+L-1})=\overline{\varepsilon}^{(L)}\}}]\mid\mathfrak{F}_{0,L}]\\
&
\nonumber
\le \kappa^L Q[D_{0,n}] \bar E_0[N_{k,n}^2-(X_{S_k}\cdot u)^2,R_{k }<\infty \mid \mathfrak{F}_{0,L}]\\
&\label{number13.1}
\le \kappa^L(1-c_{7}L^2\kappa^L)^{\left[\frac{n}{L}\right]} \bar E_0[N_{k,n}^2-(X_{S_k}\cdot u)^2,R_{k }<\infty \mid \mathfrak{F}_{0,L}],
\end{eqnarray}
\noindent where in the second inequality we have used
the Markov property, in the third one the fact that
$R_{k}\le t_{k}^{(0)}$ and in the last one Lemma \ref{qapp}.

Now, by displays (\ref{number13})
and (\ref{number13.1}), to finish the proof of inequalities (\ref{step3})
and (\ref{step3.1}) it is
 enough to prove that there  is a constant $c_{10}$ such that

\begin{eqnarray}
\nonumber
&\bar E_0[N_{k,n}^2-(X_{S_k}\cdot u)^2,R_{k }<\infty \mid \mathfrak{F}_{0,L}]\\
&
\label{newn}
\le
c_{10}\left((n+L)^2 b^{k-1}+(n+L)\bar E_0[X_{S_k}\cdot u,S_k<\infty|\mathfrak F_{0,L}]\right),
\end{eqnarray}
using the fact that $n\le L^2-1$ in the left-hand side of
inequality (\ref{step3}).
To prove (\ref{newn}), the following identity will be useful

\begin{eqnarray}
\nonumber
&N_{k,n}^2-(X_{S_{k}}\cdot u)^2= (M_k-X_{S_{k}}\cdot u)^2\\
\nonumber
&+2(n|u|_\infty+L')(M_k-X_{S_k}\cdot u)+2(n|u|_\infty+L')X_{S_k}\cdot u\\
\label{number15}
&+2(M_k-X_{S_{k}}\cdot u)X_{S_{k}}\cdot u+(n|u|_\infty+L')^2.
\end{eqnarray}
We will now insert this decomposition in the left-hand side of
 (\ref{newn}) and bound the corresponding expectations
of each term. Let us begin with the expectation
of the last term. Note that
by an argument similar to the one developed in {\it Step 1}
we have that

\begin{equation}
\label{5}
\bar E_0[(n|u|_\infty+L')^2, R_k<\infty|\mathfrak F_{0,L} ]\le
c_{11} (n+L)^2 b^k,
\end{equation}
for some constant $c_{11}$.
Similarly, the expectation of the first term of the right-hand side of
 display (\ref{number15}) can be bounded using Lemma \ref{M},
so that
\begin{eqnarray}
\nonumber
&\bar E_0[(M_k-X_{S_k}\cdot u)^2, R_k<\infty  \mid \mathfrak{F}_{0,L} ]\\
\nonumber
&=\displaystyle{\sum_{x\in H^{L}}}\mathbb{E}[\bar P_{0,\omega}[S_k<\infty, X_{S_k}=x]E_x[\mathfrak{M}^2,D'<\infty \mid \mathfrak{F}_{x,L}] \mid \mathfrak{F}_{0,L}]\\
\label{number18}
&\leq c_6 b^{k-1}.
\end{eqnarray}
Again, for the expectation of the second term of the right-hand side of display
(\ref{number15}),  we have that

\begin{eqnarray}
\nonumber
&\bar E_0[ 2(n|u|_\infty+L')(M_k-X_{S_k}\cdot u), R_k<\infty \mid \mathfrak{F}_{0,L} ]\\
&
\label{number19}
\leq c_{12} b^{k-1}(n+L),
\end{eqnarray}
for some suitable positive constant $c_{12}$.
For the expectation of fourth term of the right-hand side of (\ref{number15}),
we see by Lemma \ref{M} that

\begin{eqnarray}
\nonumber
&\bar E_0[2(M_k-X_{S_{k}}\cdot u)X_{S_{k}}\cdot u , R_k<\infty  \mid \mathfrak{F}_{0,L}]\\
\label{number20}
&\le 2\sqrt{c_6}\bar E_0[X_{S_{k}}\cdot u , S_k<\infty  \mid \mathfrak{F}_{0,L}].
\end{eqnarray}
Finally, for the expectation of third term of the
right-hand side of (\ref{number15})
we have that

\begin{eqnarray}
\nonumber
&\bar E_0[2(n|u|_\infty+L')X_{S_k}\cdot u,R_k<\infty  \mid \mathfrak{F}_{0,L}]\\\
\label{number21}
& \leq c_{12} b(n+L)  \bar E_0[X_{S_k}\cdot u,S_k<\infty \mid \mathfrak{F}_{0,L}].
\end{eqnarray}
Using the bounds (\ref{number21}), (\ref{number20}), (\ref{number19}), (\ref{number18})
and (\ref{5}) we obtain inequality  (\ref{newn}).

\medskip

\noindent \emph{ Step 4}. Here we will derive for all $k\ge 1$ the
inequality

\begin{eqnarray}
\nonumber
&\bar E_0[X_{S_k}\cdot u,S_k<\infty|\mathfrak F_{0,L}]\\
\nonumber
&\le\sum_{k'=0}^{k-1} b^{k-k'-1}\left(
\sum_{n=0}^{L^2-1} \, \bar E_0[N_{k',n}-X_{S_{k'}}\cdot u, t^{(n)}_{k' }<\infty, A_{n,k'}\mid\mathfrak{F}_{0,L} ]+\right.\\
\label{number25}
&\!\!\!\!\!\left.\sum_{n = L^2}^\infty \, \bar E_0[N_{k',n}-X_{S_{k'}}\cdot u, t^{(n)}_{k'}<\infty, B_{n,k'}, A_{n,k'}\mid \mathfrak{F}_{0,L} ]\right).
\end{eqnarray}
Note that
\begin{eqnarray}
\nonumber
&\bar E_0[X_{S_k}\cdot u, S_k<\infty \mid\mathfrak{F}_{0,L}]\\
\label{number22}
& =\sum_{k'=0}^{k-1}\bar E_0[(X_{S_{k'+1}}-X_{S_{k'}})\cdot u, S_k<\infty \mid \mathfrak{F}_{0,L}].
\end{eqnarray}
By an argument  similar to the one used
in \emph{Step 1} we see that for $k'<k$ one has that

\begin{eqnarray}
\nonumber
&\bar E_0[(X_{S_{k'+1}}-X_{S_k'})\cdot u, S_k<\infty \mid \mathfrak{F}_{0,L} ]\\
\label{number23}
&\le b^{k-k'-1}
\bar E_0[(X_{S_{k'+1}}-X_{S_k'})\cdot u, S_{k'+1}<\infty \mid \mathfrak{F}_{0,L} ].
\end{eqnarray}
Now, we can use inclusion (\ref{number9}) in order to get that
\begin{eqnarray}
\nonumber
&\bar E_0[(X_{S_{k'+1}}-X_{S_k'})\cdot u, S_{k'+1}<\infty \mid \mathfrak{F}_{0,L} ]\\
\nonumber
&\le \sum_{n=0}^{L^2-1} \, \bar E_0[(X_{S_{k'+1}}-X_{S_k'})\cdot u, t^{(n)}_{k'}<\infty, B_{n,k'}, A_{n,k'}\mid\mathfrak{F}_{0,L} ]\\
\label{number24}
&+\sum_{n = L^2}^\infty \, \bar E_0[(X_{S_{k'+1}}-X_{S_k'})\cdot u, t^{(n)}_{k'}<\infty, B_{n,k'}, A_{n,k'}\mid \mathfrak{F}_{0,L} ],
\end{eqnarray}
where the events $A_{n,k'}$ and $B_{n,k'}$ are defined in (\ref{ank})
and (\ref{bnk}).
Using the fact that on the event
$\{t^{(n)}_{k'}<\infty, B_{n,k'}, A_{n,k'}\}$ one has that
$P_0$-a.s.

$$
(X_{S_{k'+1}}-X_{S_{k'}})\cdot u \leq N_{k',n}-X_{S_k'}\cdot u,
$$
we see that the right-hand side of (\ref{number24}) is
bounded by the right-hand side of (\ref{number25}), which is what we want
to prove.

\medskip

\noindent \emph{Step 5}.
Here we will obtain an upper bound
for the terms in the first summation in (\ref{number24}).
Indeed, note  that on
 $R_{k'}\le t^{(n)}_{k'}$, by an argument similar to the one
used to derive inequality (\ref{number13}), we have that
for all $0\le n\le L^2$ and $0\le k'\le k-1$

\begin{eqnarray}
\nonumber
&  \bar E_0[N_{k',n}-X_{S_{k'}}\cdot u, t^{(n)}_{k'}<\infty, A_{n,k'} \mid \mathfrak{F}_{0,L} ]\\
\nonumber
&\le\kappa^L \bar E_0[N_{k',n}-X_{S_{k'}}\cdot u, R_{k'}<\infty \mid \mathfrak{F}_{0,L} ].
\end{eqnarray}

\medskip

\noindent \emph{Step 6}.
Here we will obtain an upper bound
for the terms in the second summation in (\ref{number24}), showing that
for all $ n\ge L^2$ and $0\le k'\le k-1$,

\begin{eqnarray}
\nonumber
&E_0[N_{k',n}-X_{S_{k'}}\cdot u, t^{(n)}_{k'}<\infty, B_{n,k'}, A_{n,k'}\mid \mathfrak{F}_{0,L} ]\\
\label{step6}
&\le \kappa^L\left(1-c_{7}L^2 \kappa^L \right)^{\left[\frac{n}{L}\right]}
\bar E_0[N_{k',n}-X_{S_{k'}}\cdot u, R_{k'}<\infty \mid \mathfrak{F}_{0,L} ].
\end{eqnarray}
Now note that
\medskip
\begin{eqnarray}
\nonumber
& E_0[N_{k',n}-X_{S_{k'}}\cdot u, t^{(n)}_{k'}<\infty, B_{n,k'}, A_{n,k'}\mid \mathfrak{F}_{0,L} ]\\
\nonumber
&\le \sum_{j= 0}^\infty\sum_{j'\geq j+n}\sum_{y\in\mathbb Z^d}E_{\mathbb{P}\otimes Q }\ [E_{\omega,\varepsilon}[N_{k',n}-X_{S_{k'}}\cdot u,\\
\nonumber
&X_{t^{(0)}_{k'}}=y,t^{(0)}_{k'}=j ] P_{\theta_y \omega,\theta_{j} \varepsilon}[
D_{j,n}, t^{(n)}_{k'}=j' ]
\mathds{1}_{\{ (\varepsilon_{j'},\ldots, \varepsilon_{j'+L-1})=\overline{\varepsilon}^{(L)}\}}]\mid\mathfrak{F}_{0,L}]\\
\nonumber
&=\kappa^L Q[D_{0,n}]
\mathbb{E}[\bar E_{0,\omega}[N_{k',n}-X_{S_{k'}}\cdot u, t^{(0)}_{k'}<\infty]
\mid \mathfrak{F}_{0,L}].
\end{eqnarray}
Using Lemma \ref{qapp} to estimate $Q[D_{0,n}]$ we conclude the proof
of inequality (\ref{step6}).

\medskip

\noindent\emph{Step 7}.
Here we will show that there exist constant $c_{13}$ and $c_{14}$ such that

\begin{equation}
\label{number28}
 \sum_{n=0}^{L^2-1}\bar E_0[N_{k',n}-X_{S_{k'}}\cdot u, t^{(n)}_{k'}<\infty, A_{n,k'} \mid \mathfrak{F}_{0,L} ]
\le c_{13} \kappa^L L^4 b^{k'-1}
\end{equation}
and

\begin{equation}
\label{number29}
 \sum_{n=L^2}^\infty \bar E_0[N_{k',n}-X_{S_{k'}}\cdot u, t^{(n)}_{k'}<\infty, A_{n,k'},B_{n,k'} \mid \mathfrak{F}_{0,L} ]\\
\le 4 c_{14}  \kappa^{-L}b^{k'-1}.
\end{equation}
Let us first note that by an argument similar to the
one used to derive the bound in Step 1 (through Lemmas \ref{lemadinf}
and \ref{M}), we have that

\begin{equation}
\label{desig}
\bar E_0[N_{k',n}-X_{S_k'}\cdot u,  R_{k'}<\infty]\\
\le (n|u|_\infty+L'+c_{15})b^{k'-1},
\end{equation}
where $c_{15}:=\sqrt{c_6}$.
Let us now prove  (\ref{number28}). Indeed, note that by
Step 5 and (\ref{desig}) we then have that
\begin{eqnarray}
\nonumber
&\sum_{n=0}^{L^2-1}\bar E_0[N_{k',n}-X_{S_{k'}}\cdot u, t^{(n)}_{k'}<\infty, A_{n,k'} \mid \mathfrak{F}_{0,L} ]\\
\nonumber
&\le\kappa^L\sum_{n=0}^{L^2-1}  \bar E_0[N_{k',n}
-X_{S_{k'}}\cdot u, R_{k'}<\infty\mid\mathfrak{F}_{0,L} ]\\
\label{number30}
&\le c_{13}\ L^4\ \kappa^L b^{k'-1},
\end{eqnarray}
for some suitable constant  $c_{13}$. Let us
now prove (\ref{number29}). First note that

\begin{eqnarray}
\nonumber
&\sum_{n=L^2}^\infty \bar E_0[N_{k',n}-X_{S_{k'}}\cdot u, t^{(n)}_{k'}<\infty, A_{n,k'},B_{n,k'} \mid \mathfrak{F}_{0,L} ]\\
\nonumber
&\le
\sum_{n=L^2}^\infty \kappa^L
(1-c_{7}L^2\kappa^L)^{\left[\frac{n}{L}\right]} \bar E_0[N_{k',n}-X_{S_{k'}}\cdot u, R_{k'}<\infty \mid \mathfrak{F}_{0,L} ]\\
\nonumber
&\le
b^{k'-1}  \sum_{n=L^2}^\infty \kappa^L
(1-c_{7}L^2\kappa^L)^{\left[\frac{n}{L}\right]}
(n|u|_\infty+L'+c_{15})
\\
\label{dddd} &
\le c_{16} b^{k'-1}
\sum_{n = L^2}^\infty n\kappa^L (1-c_{33}L^2\kappa^L)^{[\frac{n}{L^2}]}.
\end{eqnarray}
for some constant $c_{16}$, where in the first inequality we have
used \emph{Step 6} and in the second we have used inequality (\ref{desig}).
Finally notice that using
the fact that for $n\ge L^2$ one has that
 $n \leq 2L^2\left[\frac{n}{L^2}\right]$, we get that

\begin{eqnarray*}
&\sum_{n = L^2}^\infty n\kappa^L(1-c_{7}L^2\kappa^L)^{\left[\frac{n}{L^2}\right]}\leq 2\kappa^L L^2 \sum_{n =   L^2}^\infty \left[\frac{n}{L^2}\right] (1-c_{7}L^2\kappa^L)^{\left[\frac{n}{L^2}\right]} \\
&=2L^4 \kappa^L \sum_{m=1}^\infty m(1-c_{7}L^2\kappa^L)^m \leq
 \frac{2}{(c_{7})^2}\kappa^{-L}.
\end{eqnarray*}
Using this estimate in (\ref{dddd}) we obtain (\ref{number29}).

\medskip

\noindent \emph{Step 8}. Here we finish the proof
of Proposition \ref{prop-sec5} combining the previous steps we
have already developed. Combining inequality (\ref{number25})
proved in  \emph{Step 4} with inequalities
(\ref{number28}) and (\ref{number29}) proved in \emph{Step 7}, we see that
there is a constant $c_{17}$ such that

\begin{equation}
\label{number33}
\bar E_0[X_{S_{k}}\cdot u, S_k<\infty \mid \mathfrak{F}_{0,L} ]\leq
c_{17} kb^{k-2}  \kappa^{-L}.
\end{equation}
Thus, by inequality (\ref{step3}) proved in \emph{Step 3}, we have that
\begin{equation}
\label{number34}
\sum_{n=0}^{L^2-1}\bar E_0[(X_{S_{k+1}}\cdot u)^2-(X_{S_{k}}\cdot u)^2,t^{(n)}_{k }<\infty,  A_{n,k}  \mid \mathfrak{F}_{0,L}]\leq
c_{18} L^4 k b^{k-2}.
\end{equation}
for certain positive constant $c_{18}$
On the other hand, combining  inequality (\ref{step3.1}) proved
in Step 3 with (\ref{number33}), we see that there exists
a constant $c_{19}$ such that
\begin{eqnarray}
\nonumber
&\sum_{n = L^2}^\infty\bar E_0[(X_{S_{k+1}}\cdot u)^2-(X_{S_{k}}\cdot u)^2,t^{(n)}_{k}<\infty, B_{n,k}, A_{n,k}  \mid \mathfrak{F}_{0,L}]\\
\nonumber
&\le c_{19}\sum_{n=L^2}^\infty
 \kappa^L (1-c_{7}L^2\kappa^L)^{\left[\frac{n}{L^2}\right]}
\left((n+L)^2 b^{k-1}\right.\\
\label{number35}
&\left. +(n+L)kb^{k-2}\kappa^{-L}\right).
\end{eqnarray}
Now, note that for some constant $c_{20}$ one has that

\begin{eqnarray}
\label{sum1}
&\sum_{n= L^2}^\infty (n+L)^2(1-c_{7}L^2\kappa^L)^{[\frac{n}{L^2}]}\leq c_{20}  \ \kappa^{-3L}\quad {\rm and}\\
\label{sum2}
&
\sum_{n= L^2}^\infty (n+L)(1-c_{7}L^2\kappa^L)^{[\frac{n}{L^2}]}\leq c_{20}  \ \kappa^{-2L}.
\end{eqnarray}
Substituting (\ref{sum1}) and (\ref{sum2}) into (\ref{number35})
we see that
\begin{eqnarray}
\nonumber
&\sum_{n= L^2}^\infty\bar E_0[(X_{S_{k+1}}\cdot u)^2-(X_{S_{k}}\cdot u)^2,t^{(n)}_{k}<\infty, B_{n,k}, A_{n,k}  \mid \mathfrak{F}_{0,L}]\leq\\
\label{number35.1}
&c_{21}  \kappa^{-2L}b^{k-2} k,
\end{eqnarray}
for some suitable positive constant $c_{21}$.
Substituting (\ref{number35}) and (\ref{number35.1}) into
inequality (\ref{number11}) of Step 2, we then conclude that
there is a constant $c_{22}$ such that
\begin{equation}
\label{number1111}
\bar E_0[(X_{S_{k+1}}\cdot u)^2-(X_{S_{k}}\cdot u)^2, S_{k+1}<\infty|\mathfrak F_{0,L}]\le  c_{22}\kappa^{-2L}b^{k-2} k.
\end{equation}
Substituting (\ref{number1111}) into (\ref{number7.1}) of Step 1, we get that

\begin{eqnarray}
\nonumber
&\bar E_0[(X_{S_{k'+1}}\cdot u)^2-(X_{S_{k'}}\cdot u)^2, S_k<\infty,
 D'\circ\theta_{S_k}=\infty\mid \mathfrak{F}_{0,L}]\\
\label{number7.1.1}
&\le b' b^{k+1}k'.
\end{eqnarray}
From the fact that
$\sum_{k=1}^\infty\sum_{k'=0}^{k-1}b^{k+1}k'<\infty$
together with (\ref{number7.1.1} and (\ref{decomXt}) of \emph{Step 0},
we conclude that

$$
\bar E_0[(X_{\tau}\cdot u)^2|\mathfrak F_{0,L}]\le c_{23} \kappa^{-2L},
$$
for some constant $c_{23}>0$,
which proves the proposition.

\medskip

 \section{Proof of Theorem \ref{mainth}}
\label{finals}

In this section we will prove  Theorem \ref{mainth} using
Proposition \ref{prop-sec5} proved in Section \ref{section5}.
First in Subsection \ref{subsection6.1}, we will define
an approximate sequence of regeneration times. In
Subsection \ref{subsection6.2}, we will show through
this approximate regeneration time sequence, that
there exists an approximate asymptotic direction.
In Subsection \ref{subsection6.3}, we will use
the approximate asymptotic direction to prove
Theorem \ref{mainth}.
\medskip

\subsection{Approximate regeneration time sequence}
\label{subsection6.1}
As in \cite{CZ01}, we define  approximate regeneration
by the recursively by $\tau_1^{(L)}:=\tau$ [c.f.(\ref{taudef})]
and for $i\ge 2$

$$
\tau_i^{(L)}:=\tau_1^{(L)}\circ\theta_{\tau_{i-1}^{(L)}}+\tau_{i-1}^{(L)}.
$$
We will drop the dependence in $L$ on $\tau_1^{(L)}$ when it is convenient for us, using the notation $\tau_i$ instead $\tau_i^{(L)}$.
Let us define  $\sigma$-algebras
corresponding to the information of the random walk
and the $\varepsilon$ process up to the  first
regeneration time and
of the environment $\omega$ at a distance of order $L$ to the left of
the position of the random walk at this regeneration time as

\begin{eqnarray}
\nonumber
&\mathcal{H}_1:=\sigma(\tau_{1}^{(L)},X_0,\varepsilon_0,\ldots,\varepsilon_{\tau_{1}^{(L)}-1},X_{\tau_{1}^{(L)}},\\
\nonumber
&\{\omega(y, \cdot): \ y\cdot u<u\cdot X_{\tau_1^{(L)}}-L|u|/|u|_1\}).
\end{eqnarray}
Similarly define for $k\ge 2$

\begin{eqnarray}
\nonumber
&\mathcal{H}_k:=\sigma(\tau_{1}^{(L)},\ldots,\tau_{k}^{(L)},X_0,\varepsilon_0,\ldots,\varepsilon_{\tau_{k}^{(L)}-1},X_{\tau_{k}^{(L)}},\\
\label{Hk}
&\{\omega(y, \cdot): \ y\cdot u<u\cdot X_{\tau_{k}^{(L)}}-L|u|/|u|_1\} ).
\end{eqnarray}
Let us now recall Lemma 2.3  of \cite{CZ01}, stated
here under the condition $P_0[D'=\infty]>0$  [c.f. (\ref{defdp})] instead of
Kalikow's condition.

\medskip

\begin{lemma}
\label{lemmaC-Z} Let $l\in\mathbb{S^*}^{d-1}$, $\alpha>0$ and
$\phi$ be such that $\lim_{r\to\infty}\phi(r)=0$.
Consider a random walk in a random environment
satisfying the cone-mixing assumption with respect to $\alpha$, $l$
and $\phi$ and uniformly elliptic
with respect to $l$.
Assume that $L$ is such that

$$
\phi(L)< P_0[D'=\infty].
$$
Then, $\mathbb{P}-$a.s. one has that

$$
\left|\bar P_0[\{X_{\tau_k +\cdot}-X_{\tau_k}\}\in A \mid \mathcal{H}_k]-\bar P_0[\{X_\cdot \}\in A|D'=\infty]\right| \leq \phi'(L),
$$
for all measurable sets $A\subset\mathbb (\mathbb Z^d)^{\mathbb N}$,
where

$$
\phi'(L):=\frac{2 \phi(L)}{(P_0[D'=\infty]-\phi(L))}.
$$

\end{lemma}

\medskip
\begin{proof}

For $k=1$, the argument given in page 890 of (\cite{CZ01}) still works without any change. With the purpose of showing that the result continues being
true under the \emph{weaker} assumptions here, we complete the induction argument in the case $k=2$.
To this end, we consider a positive $\mathcal{H}_2-$ measurable function $h$ of the form $h=h_1 \cdot (h_2)\circ\theta_{\tau_1}$ ($\cdot$ denotes usual function multiplication), such that $h_1,$ is $\mathcal{H}_1-$ measurable and $h_2$ is $\mathcal{H'}_1$ measurable, where the $\sigma-$ algebra $\mathcal{H'}_1$ is defined as :
\begin{eqnarray*}
\mathcal{H'}_1:&=&\sigma(\tau_{1}^{(L)},X_0,\varepsilon_0,\ldots,\varepsilon_{\tau_{1}^{(L)}-1},X_{\tau_{1}^{(L)}},\\
&&\{\omega(y, \cdot): \ u \cdot y \leq u\cdot X_{\tau_1^{(L)}}-L\frac{|u|}{|u|_1},\ y\in C(X_0,l,\alpha) \}).
\end{eqnarray*}
We let $A$ be a measurable set of the path space,
for short we will write $\mathds{1}_{A}:=\mathds{1}_{\{(X_n-X_0)_{n\geq 0}\in A\}}$. By the strong Markov property and using that $\tau_1<\infty$ within an event of full $P_0$ probability, we get:
\begin{eqnarray}
\nonumber
&\displaystyle{\bar E_{0}[h\mathds{1}_A \circ \theta_{\tau_ 2}]\leq \sum _{n\geq1}E_{0}[h\mathds{1}_A \circ \theta_{\tau_ 2} ]}\\
\nonumber
&\displaystyle{\bar E_{0}[h\mathds{1}_A \circ \theta_{\tau_ 2}\ \mathds{1}_{K=n}\circ \theta_{\tau_{1}}, \tau_1<\infty]}\\
\nonumber
&\displaystyle{\sum_{  t\geq 1 }\bar E_{0}[h\mathds{1}_A \circ \theta_{\tau_ 2} \ , \tau_1=t]}\\
\label{firstexp}
&\displaystyle{\sum_{  t\geq 1 }\bar E_{0}[h_1 \cdot (h_2\circ\theta_{\tau_1}) \mathds{1}_A \circ \theta_{\tau_ 2}\ , S_t<\infty, D'\circ \theta_{S_t}=\infty]}.
\end{eqnarray}
Now, notice that for given $t\in \mathbb{N}, m \in \mathbb{N}, x\in \mathbb{Z}^d$, we can find a random variable $h_{1,t, m, x }$
measurable with respect to $\sigma (\{\omega(y,\cdot): y\cdot u <x\cdot u-L\frac{|u|}{|u|_1}\}, \{X_i\}_{i<m})$ such that it coincides with $h_1$ on the event $\{\tau_1 =S_t=m, X_{S_t}=x\}$, therefore (\ref{firstexp}) equals
\begin{eqnarray}
\nonumber
&\displaystyle{\sum_{ \ t\geq 1 , m\geq 1, x\in \mathbb{Z}^d }\bar E_{0}[h_{1,t, m, x } (h_2\circ\theta_{\tau_1}) \mathds{1}_A \circ \theta_{\tau_ 2}\ \mathds{1}_{ S_t=m, D'\circ \theta_{m}=\infty, X_m=x}]}\\
\nonumber
&\displaystyle{\sum_{ \ t\geq 1 , m\geq 1, x\in \mathbb{Z}^d }\bar E_{0}[h_1\mathds{1}_{S_t=m, X_m=x D'\circ \theta_{m}=\infty}\mathds{1}_A \circ \theta_{\tau_ 2}h_2\circ\theta_{\tau_1}]}\\
\nonumber
&\displaystyle{\sum_{ \ t\geq 1 , m\geq 1, x\in \mathbb{Z}^d }E_{\mathbb{P} \otimes Q}[E_{0,\omega, \varepsilon}[h_{1,t, m, x } \mathds{1}_{S_t=m, X_m=x} \mathds{1}_{ D'\circ \theta_{m}=\infty} \mathds{1}_A \circ \theta_{\tau_ 2}h_2\circ\theta_{\tau_1}]]}\\
\nonumber
&\displaystyle{\sum_{ t\geq 1 , m\geq 1, x\in \mathbb{Z}^d }E_{\mathbb{P}\otimes Q}[E_{0,\omega, \varepsilon}[h_{1,t, m, x }\mathds{1}_{S_t=m, X_m=x}]E_{\theta_x\omega,\theta_m\varepsilon}[\mathds{1}_{D'=\infty} \mathds{1}_A \circ \theta_{\tau_ 1}h_2]]}.
\end{eqnarray}
We now work out the following expression
\begin{eqnarray}
\nonumber
&E_{\theta_x\omega,\theta_m\varepsilon}[\mathds{1}_{D'=\infty} \mathds{1}_A \circ \theta_{\tau_ 1}h_2]\\
\label{exprf}
&\displaystyle{\sum_{\substack{z\in C(x,l,\alpha)\\ n\geq1, j \geq m+ 1}}E_{\theta_x\omega,\theta_m\varepsilon}[\mathds{1}_{D'=\infty} \mathds{1}_A \circ \theta_{\tau_ 1}h_2, S_n=j, X_{S_n}=z, D'\circ\theta_j=\infty]}.
\end{eqnarray}
Observe that, as in the case of $h_1$, for fixed $x$ and $m$, we consider the probability measure $P_{\theta_x\omega,\theta_m\varepsilon}$. Then we can find
a measurable function $h_{2,j,n,z}$ with respect to $\sigma (\{\omega(y,\cdot):  y\cdot u  \leq z\cdot u-L\frac{|u|}{|u|_1}, \ y \in C(x,l,\alpha)\}, \{X_i\}_{i<j})$ , which coincides with $h_2$ on the event $\{\tau_1=S_n =j, X_{S_n}=z, D'=\infty\}$, furthermore note that $D'=\infty$ depends up to $(j-1)$ coordinate in $\varepsilon$ (recall that $\{D'=\infty\}\in \mathcal{H}_1$), hence we can apply the Markov property to get that the last expression in (\ref{exprf}) is equal to:
\begin{equation}
\label{expansion}
\sum_{\substack{z\in C(x,l,\alpha)\\ n\geq1, j \geq 1}}E_{\theta_x\omega,\theta_m\varepsilon}[h_{2,j,n,z},\mathds{1}_{ S_n =j, X_{S_n}=z, D'=\infty}]P_{\theta_z\omega, \theta_j\varepsilon}[A\cap\{D'=\infty\}].
\end{equation}
Using (\ref{expansion}), it follows that (\ref{firstexp}) is equal to:
\begin{eqnarray*}
\nonumber
&\displaystyle{\sum_{ t\geq 1 , m\geq 1, x \in \mathbb{Z}^d }\sum_{\substack{z\in C(x,l,\alpha)\\ n\geq1, j \geq m+ 1}}E_{\mathbb{P} \otimes Q}[E_{0,\omega, \varepsilon}[h_{1,t, m, x }\mathds{1}_{S_t=m, X_m=x}]\cdot}\\
\nonumber
&\displaystyle{E_{\theta_x\omega,\theta_m\varepsilon}[h_{2,j,n,z},\mathds{1}_{ S_n =j, X_{S_n}=z, D'=\infty}]P_{\theta_z\omega, \theta_j\varepsilon}[A\cap\{D'=\infty\}]]}
\end{eqnarray*}
Following \cite{CZ01}, we can write down the expression above as
\begin{eqnarray}
\nonumber
&\displaystyle{\sum_{ t\geq 1 , m\geq 1, x \in \mathbb{Z}^d }\sum_{\substack{z\in C(x,l,\alpha)\\ n\geq1, j \geq m+ 1}} E_{\mathbb{P} \otimes Q}[[E_{0,\omega, \varepsilon}[h_{1,t, m, x }\mathds{1}_{S_t=m, X_m=x}]]\cdot}\\
\nonumber
&\displaystyle{E_{\theta_x\omega,\theta_m\varepsilon}[h_{2,j,n,z}\mathds{1}_{ S_n =j, X_{S_n}=z, D'=\infty}]] \bar P_0[A\cap\{D'=\infty\}]+ \rho(A)},
\end{eqnarray}
where
$$
\rho(A):=\sum_{ t\geq 1 , m\geq 1, x \in \mathbb{Z}^d }\sum_{\substack{z\in C(x,l,\alpha)\\ n\geq1, j \geq m+ 1}}Cov_{\mathbb{P}\otimes Q}[f_{t,m,x,j,n,z},g_{j,z}],
$$
with:
$$
f_{t,m,x,j,n,z}:=E_{0,\omega, \varepsilon}[h_{1,t, m, x }\mathds{1}_{S_t=m, X_m=x}] E_{\theta_x\omega,\theta_m\varepsilon}[h_{2,j,n,z},\mathds{1}_{ S_n =j, X_{S_n}=z, D'=\infty}]
$$
and
$$
g_{j,z}:=P_{\theta_z\omega, \theta_j\varepsilon}[A\cap\{D'=\infty\}].
$$
On the other hand, since assumption $(CM)_{\phi,\alpha}|l$, the estimate
\begin{eqnarray}
\nonumber
&\displaystyle{\rho(A)\leq \phi(L) \sum_{ t\geq 1 , m\geq 1, x \in \mathbb{Z}^d }\sum_{\substack{z\in C(x,l,\alpha)\\ n\geq1, j \geq m+ 1}} E_{\mathbb{P} \otimes Q}[E_{0,\omega, \varepsilon}[h_{1,t, m, x }\mathds{1}_{S_t=m, X_m=x}]\cdot}\\
\nonumber
&\displaystyle{E_{\theta_x\omega,\theta_m\varepsilon}[h_{2,j,n,z}\mathds{1}_{ S_n =j, X_{S_n}=z, D'=\infty}]]}
\end{eqnarray}
holds for all measurable set $A$ in the path space, in particular applying this for $A=\mathbb{Z}^d$ turns out the estimate:
\begin{eqnarray}
\nonumber
&\displaystyle{\sum_{ t\geq 1 , m\geq 1, x \in \mathbb{Z}^d }\sum_{\substack{z\in C(x,l,\alpha)\\ n\geq1, j \geq m+ 1}} E_{\mathbb{P} \otimes Q}[E_{0,\omega, \varepsilon}[h_{1,t, m, x }\mathds{1}_{S_t=m, X_m=x}]\cdot}\\
\nonumber
&\displaystyle{E_{\theta_x\omega,\theta_m\varepsilon}[h_{2,j,n,z}\mathds{1}_{ S_n =j, X_{S_n}=z, D'=\infty}]]\leq}\\
\nonumber
&(P_0[D'=\infty]-\phi(L))^{-1}\bar E_0[h].
\end{eqnarray}
From now on, we can follow the same sort of argument as in (\cite{CZ01}), in order to conclude that
\begin{equation*}
\parallel \bar P_0[\{X_{\tau_2 +n}-X_{\tau_2}\}\in\cdot \mid \mathcal{H}_2]-\bar P_0[\{X_n\}\in\cdot\mid D'=\infty]\parallel_{var} \leq \phi'(l).
\end{equation*}
Therefore the second step induction is complete.
\end{proof}

\medskip

\subsection{Approximate asymptotic direction}
\label{subsection6.2}
We will show that a random satisfying the
cone mixing, uniform ellipticity assumption and the
non-effective polynomial condition with high enough
degree has an approximate asymptotic direction.
The exact statement is given below. It will also
be shown that the right order in which the random
variable $X_{\tau_1}$ grows as a function of $L$ is
$\kappa^{-L}$.

\begin{proposition}
\label{mp}
 Let $l\in\mathbb{S^*}^{d-1}$, $\phi$
be such that $\lim_{r\to\infty}\phi(r)=0$, $c>0$, $M>6d$
and $0<\alpha<\min\{\frac{1}{9},\frac{1}{2c+1}\}$.
Consider a random walk in a random environment
satisfying the cone mixing condition with respect to
$\alpha$, $l$ and $\phi$  and
the uniform ellipticity condition with respect to $l$.
Assume that $(PC)_{M,c}|l$ is satisfied.
Then, there exists a sequence $\eta_L$ such that
$\lim_{L\to\infty}\eta_L=0$ and $\bar P_0$-a.s.

\begin{equation}
\label{e1p62}
\limsup_{n\rightarrow\infty}\left|\frac{\kappa^L X_{\tau_n}}{n}-\lambda_L \right|<\eta_L,
\end{equation}
where for all $L\ge 1$,

\begin{equation}
\label{lambda}
\lambda_L: =\bar E_0[\kappa^L X_{\tau_1} \mid D'=\infty].
\end{equation}
Furthermore,

\begin{equation}
\label{lambda2}
|\lambda_L|_2\ge c_{270}\kappa^{-L},
\end{equation}
for some constant $c_{270}$.
\end{proposition}

\medskip

\noindent We first prove inequality (\ref{e1p62}) of Proposition \ref{mp}.
We  will follow
the argument presented for the proof of Lemma 3.3 of
 \cite{CZ01}.
For each integer $i\geq 1$ define the sequence

$$
\overline{X}_i:=\kappa^L (X_{\tau_i}-X_{\tau_{i-1}}),
$$
with the convention $\tau_0=0$.
Using  Lemma \ref{lemmaC-Z} and
Lemma 3.2 of \cite{CZ01},
 we can enlarge the probability space where
 the sequence $\{X_i:i\ge 1\}$ so that
there we have the following properties:

\begin{itemize}
  \item[(1)] There exist an i.i.d. sequence $\{(\widetilde{X}_i,  \Delta_i):i\ge 2\}$ of random vectors
 with values in $(\kappa^{L}\mathbb{Z}^d,  \{0,1\})$, such that
 $\widetilde{X}_2$ has the same distribution as
$\overline{X}_1$ under the measure $\bar P_0[\cdot | D'=\infty]$
while $\Delta_2$ has a Bernoulli distribution on $\{0,1\}$ with $\bar P_0[\Delta_i=1]=\phi'(L)$.

\item[(2)] There exists a sequence $\{Z_i:i\ge 2\}$ of random
variables such that for all $i\ge 2$ one has that

\begin{equation}
\label{oxi}
  \overline{X}_i=(1-\Delta_i)\widetilde{X}_i+\Delta_i Z_i.
 \end{equation}
Furthermore, for each $i\ge 2$,   $\Delta_i$ is independent of
$Z_i$ and of

 $$
\mathcal{G}_i:=\sigma \{ \overline{X}_j:j \leq i-1\}.
$$
\end{itemize}

\medskip

\noindent We will
call $P$ the common probability distribution
of the sequences $\{\overline X_i:i\ge 2\}$, $\{\widetilde X_i:i\ge 2\}$,
$\{Z_i:i\ge 2\}$ and $\{\Delta_i:i\ge 2\}$, and $E$ the
corresponding expectation.
From (\ref{oxi}) note that

\begin{equation}
\label{sumx}
\frac{1}{n}\sum_{i=1}^n\overline{X}_{i}
=\frac{\overline{X}_1}{n}+\frac{1}{n}\sum_{i=2}^n \widetilde{X}_{i}-\frac{1}{n}
\sum_{i=2}^n \Delta_i\widetilde{X}_{i}
+\frac{1}{n}\sum_{i=1}^n \Delta_i Z_i.
\end{equation}
Let us now examine the behavior as $n\to\infty$ of each
of the four terms in the left-hand side of (\ref{sumx}). Clearly, the first term tends
to $0$ as $n\to \infty$.
For the second term, note that on the event $\{D'=\infty\}$, one has that
$\mid\overline{X}_1\mid_2 ^2\le c_{24}(\overline{X}_1\cdot l)^2$ for some constant
$c_{24}$. Therefore, by Proposition \ref{prop-sec5},
and the fact that $\widetilde {X}_2$ has the same distribution as
$\overline{X}_1$ under $\bar P_0[\cdot |D'=\infty]$, we see that

\begin{equation}
\label{ex1d}
E[|\widetilde{X}_2|_2 ^2]=\bar E_0[|\overline{X}_1|_2 ^2|D'=\infty]\le
c_{24}\bar E_0[(\overline{X}_1\cdot l)^2|D'=\infty]
<c_{25},
\end{equation}
for a suitable constant $c_{25}$.
Hence, by the strong law of large numbers, we actually have that $P$-a.s.

\begin{equation}
\label{ndelta4}
\lim_{n\rightarrow\infty}\frac{1}{n}
\sum_{i=2}^n \widetilde{X}_{i}=\lambda_L.
\end{equation}
For the third term in the left-hand side of (\ref{sumx}) we
have by Cauchy-Schwartz inequality that

\begin{equation}
\label{ndelta}
\left|\frac{1}{n}\sum_{i=2}^n  \Delta_i \widetilde{X}_{i}\right|_2\leq
 \left(\frac{1}{n} \sum_{i=2}^{n} | \widetilde{X}_{i}|^{2}\right)^{\frac{1}{2}} \left(\frac{1}{n}\sum_{i=2}^n \Delta_i\right)^{\frac{1}{2}}.
\end{equation}
Again by (\ref{ex1d}) and Proposition \ref{prop-sec5}, we know
that there is a constant $c_{26}$ [c.f. (\ref{eq-prop5})]
such that $P$-a.s.

$$
\lim_{n\to\infty}\frac{1}{n}\sum_{i=2}^n |\widetilde{X}_i|_2^2=\bar E_0[|\overline{X}_1|_2 ^2|
D'=\infty]\le c_{26}.
$$
As a result, from (\ref{ndelta}) we see that

\begin{equation}
\label{ndelta2}
\limsup_{n\rightarrow\infty}\left|\frac{1}{n}
\sum_{i=2}^n \Delta_i \widetilde{X}_{i}\right|_2\leq \sqrt{c_{26}
\phi'(L)}.
\end{equation}
For the fourth term of the left-hand side of
(\ref{sumx}), we note
setting $\overline{Z}_i^{(L)}:= E[Z_i\mid \mathcal{G}_i]$
that

$$
M_n^j:=\sum_{i=2}^n \frac{\Delta_i(Z_i-\overline{Z}_i)\cdot e_j}{i}\quad {\rm for}\ n\ge 2, \ j\in\{1,2,\ldots, n\}
$$
is a  martingale with mean zero with respect to the filtration $\{\mathcal{G}_i:i\geq 1\}$.
 Thus, from the Burkholder-Gundy inequality \cite{W91}, we know
that there is a constant $c_{27}$ such that for all $j\in\{1,2,\ldots, d\}$

\begin{equation}
\label{bgm}
E\left[\left(\sup_n M_n^j\right)^{2}\right]\leq c_{27} E \left[
\sum_{i=2}^{\infty}\frac{|\Delta_i(Z_i-\overline{Z}_i)|_2 ^2}{i^2}
\right].
\end{equation}
Now, since (\ref{oxi}), note that for all $i\ge 2$,
$|\Delta_i Z_i|\le |\bar X_i|$. It follows that
there exists a constant $c_{28}$ such that

\begin{equation}
\label{zuno2}
E[|Z_i|_2 ^2|\mathcal G_i]\le \frac{1}{\phi'(L)} E_0[|\overline{X}_1|_2 ^2, D'=\infty|\mathfrak F_{0,L}]\le \frac{1}{\phi'(L)} c_{28},
\end{equation}
where we have used Proposition \ref{prop-sec5} and Lemma \ref{lemadinf}
in the second inequality. So that by (\ref{bgm}) we see that the martingale
 $\{M_n^j:n\ge 1\}$ converges $P-$a.s. to a random variable for any $j\in\{1,2,\ldots,d\}$.
Thus, by Kronecker's lemma applied to each component $j\in\{1,2,\ldots,d\}$, we conclude that $P$-a.s.

\begin{equation}
\label{zkron}
\lim_{n\rightarrow\infty}\frac{1}{n}\sum_{i=2}^n\Delta_i(Z_i-\overline{Z}_i)=0.
\end{equation}
Now, note from (\ref{zuno2}) that there is a constant $c_{29}$ such that

\begin{equation}
\label{zbari}
|\overline{Z}_i|_2\le E[|Z_i|_2 ^2 \mid \mathcal{G}_i]^{\frac{1}{2}}\leq c_{29}\phi'(L)^{ -\frac{1}{2}}.
\end{equation}
Therefore, $P$-a.s. we have that

\begin{eqnarray}
\nonumber
&\limsup_{n\to\infty}\left|\frac{1}{n}\sum_{i=2}^n \Delta_i\overline{Z}_i
 \right|_2\leq c_{29}\phi'(L)^{ -\frac{1}{2}}\limsup_{n\to\infty} \frac{1}{n}\sum_{i=1}^n \Delta_i\\
\label{ndelta3}
&\le c_{29} \phi'(L)^{\frac{1}{2}}.
\end{eqnarray}
Substituting (\ref{ndelta3}), (\ref{ndelta2}) and (\ref{ndelta4}) into
(\ref{sumx}), we conclude the proof of inequality (\ref{e1p62})
provided we set $\eta_{L}=c_{30}\phi'(L)^{\frac{1}{2}}$
for some constant $c_{30}$.

Let us now prove the inequality  (\ref{lambda2}). By an argument
similar to the one presented in \cite{CZ01} to show that the random
variable $\tau_1$ has a lower bound of order $\kappa^{-L}$, we
can show that $X_{\tau_1}\cdot l$ is bounded from below by the sum
$S:=\sum_{i=1}^N U_i$, where $\{U_i:i\ge 1\}$ are i.i.d. random
variables taking values on $\{1,2,\ldots, L\}$ with
law $P[U_i=n]=\kappa^n$ for $1\le n\le L$, while $N:=\min\{ i\ge 1: U_i=L\}$.
It is clear then that

$$
 E[X_{\tau_1}\cdot l] \ge E[N]=c_{31}\kappa^{-L},
$$
for some constant $c_{31}$.
\medskip

\subsection{Proof of  Theorem \ref{mainth}}
\label{subsection6.3}
It will be enough to prove that there is a constant
$c_{32}$ such that for all $L\ge 1$ one has that
\begin{equation}
\label{ineq}
\limsup_{n\rightarrow\infty}\left|\frac{X_{n}}{|X_{n}|_2}-\frac{\lambda_L}{|\lambda_L|_2}\right|_2< c_{260} \frac{\eta_L}{\lambda_L}.
\end{equation}
Indeed, by compactness, we know that we can choose
a sequence $\{L_m, m\ge 1\}$ such that

\begin{equation}
\label{ineqq}
\lim_{m\to\infty}\frac{\lambda_{L_m}}{|\lambda_{L_m}|_2}=\hat v,
\end{equation}
exists. On the other hand, by the inequality  (\ref{lambda2})
of Proposition \ref{prop-sec5}, we know that
$\lim_{m\to\infty}\frac{\eta_{L_m}}{\lambda_{L_m}}=0$.
Now note that by the triangle inequality and (\ref{ineq}), for every $m\ge 1$ one has that

\begin{equation}
\label{ineqqq}
\limsup_{n\rightarrow\infty}\left|\frac{X_{n}}{|X_{n}|_2}-\hat v\right|_2\le
c_{32} \frac{\eta_{L_m}}{\lambda_{L_m}}+
\left|\frac{\lambda_{L_m}}{|\lambda_{L_m}|_2}-\hat v\right|_2.
\end{equation}
Taking the limit $m\to\infty$ in (\ref{ineqqq}) using (\ref{ineqq})
we prove Theorem \ref{mainth}.

Let us hence prove inequality (\ref{ineq}).
Choose  a nondecreasing sequence $\{k_n:n\ge 1\}$, $P$- a.s. tending to $+\infty$  so that
for all $n\ge 1$ one has that
$$
\tau_{k_n} \leq n < \tau_{k_n +1}.
$$
Notice that
\begin{equation}
\label{decom}
\frac{X_{n}}{|X_{n}|_2}=\left(\frac{X_{n}-X_{\tau_{k_{n}}}}{|X_{n}|_2}\right)+\left(\frac{X_{\tau_{k_{n}}}}{k_n}\frac{k_n}{|X_{n}|_2}\right).
\end{equation}
On the other hand, we assume for the time being,
that for  large enough $L$ we have proved that
\begin{equation}
\label{ineq1}
\limsup_{n\rightarrow\infty}\frac{|X_{n}-X_{\tau_{k_{n}}}|_2}{k_{n}} =0.
\end{equation}
Note first that (\ref{ineq1}) implies that

\begin{equation}
\label{ineq12}
\limsup_{n\rightarrow\infty}\frac{|X_{n}-X_{\tau_{k_n}}|_2}{|X_{n}|_2}=0.
\end{equation}
Indeed, note that
$|X_n|_2\ge X_n \cdot l \geq X_{\tau_{k_n}}\cdot l\geq k_n L
\frac{|l|_2} {|l|_1}$, which in combination with (\ref{ineq1})
implies (\ref{ineq12}).
Also,  from (\ref{ineq1}) and the fact that

\begin{equation}
\label{ineq2}
\frac{|X_{\tau_{k_{n}}}|_2}{k_n}-\frac{|X_{n}-X_{\tau_{k_{n}}}|_2}{k_{n}}\leq\frac{|X_n|_2}{k_{n}}\leq
\frac{|X_{\tau_{k_{n}}}|_2}{k_n}+\frac{|X_{n}-X_{\tau_{k_{n}}}|_2}{k_{n}},
\end{equation}
we see that

\begin{equation}
\label{ineq4}
\limsup_{n\rightarrow\infty}\left|\frac{\kappa^L|X_n|_2}{k_n}-|\lambda_L|_2\right|_2\leq \eta_L.
\end{equation}

Combining (\ref{ineq12}) and (\ref{ineq4}) with (\ref{decom}) we get (\ref{ineq}).
Thus, it is enough to prove the claim in (\ref{ineq1}). To this end,
note that

\begin{eqnarray}
\nonumber
&\displaystyle{\frac{|X_{n}-X_{\tau_{k_{n}}}|_2}{k_{n}} \leq \sup_{j \geq 0}\frac{|X_{(\tau_{k_n}+j)\wedge \tau_{k_n +1}}-X_{\tau_{k_{n}}}|_2}{k_{n}}}\\
\label{final1}
\end{eqnarray}
We now consider the sequence $\widehat{X}_{k\geq 1}:=\left(\kappa^L\ \sup_{j\geq0}|X_{(\tau_{k}+j)\wedge(\tau_{k+1})}-X_{\tau_k}|\right)_{k\geq 1}$, a coupling decomposition as in the proof of Proposition \ref{mp} turns out; in a enlarged probability space $\mathrm{P}$ if necessary, the existence of two i.i.d. sequences $\left(\mathrm{X}_k\right)_{k\geq1}, \ \left(\mathrm{\Delta}_k\right)_{k\geq 1}$ and a sequence $\left(Y_k\right)_{k\geq 1}$, such that $\mathrm{P}$ supports the following:
\begin{itemize}
\item For $k\geq1$, the common law of $\mathrm{X}_k$ is the same as $\widehat{X}_1$ under $\bar P[\cdot\mid D'=\infty]$, and one has that $\mathrm{\Delta}_k$ is Bernoulli with values in the set $\{0,1\}$ independent of $\mathcal{G}_k$ and $\mathrm{P}[\mathrm{\Delta}_k=1]=\phi'(L)$.
\item $\mathrm{P}$- almost surely for $k\geq 1$, we have the decomposition:
$$
\widehat{X}_k=(1-\mathrm{\Delta}_k)\mathrm{X}_k+\mathrm{\Delta}_k Y_k
$$
\end{itemize}
Furthermore, quite similar arguments as the ones given in the proof of Proposition \ref{mp} allow us to conclude that:

\begin{eqnarray}
\nonumber
&\displaystyle{\sum_{j=1}^n \frac{|\mathrm{X}_j|}{n}\rightarrow \mathrm{E}[|\widehat{X_1}|\mid D'=\infty]<\infty},\\
\nonumber
&\displaystyle{\sum_{j=1}^n\frac{\mathrm{\Delta}_j (Y_j-\widetilde{Y}_j)}{n}\rightarrow0}\mbox{  and}\\
\label{fs}
&\displaystyle{\sum_{j=1}^n\frac{\mathrm{|\Delta}_j \widetilde{Y}_j|}{n}\leq c_{240}\phi'(L)^{\frac{1}{2}}}.\\
\end{eqnarray}
where $\widetilde{Y}_j:=\mathrm{E}[Y_j\mid\mathcal{G}_j]$. Therefore, using the following inequality
\begin{equation}
\frac{\widehat{X}_k}{k}=\frac{\mathrm{X}_k}{k}+\frac{\mathrm{\Delta}_k (Y_k-\widetilde{Y}_k)}{k}+\frac{\mathrm{\Delta}_k \widetilde{Y}_k}{k},
\end{equation}
implies that
\begin{equation}
\frac{\mathrm{X}_k}{k}\rightarrow_{\substack{k \to \infty}}0
\end{equation}
The proof is finished.

\appendix
\section{Proof of Lemma \ref{qapp}}
\label{appendix}

Here we will prove Lemma \ref{qapp}.
Let us first remark that it will be enough to
show that there exists a constant $c_{33}>0$ such that for all
$L\in |u|_1\mathbb N$

\begin{equation}
\label{Qn}
Q[D_{0,L^2}]\leq 1-c_{33}L^2\kappa^L.
\end{equation}
Indeed, using this inequality and the product structure of $Q$,
for all $n\ge L^2$ one has that

$$
Q[D_{0,n}]\le
(1-c_{7}L^2\kappa^L)^{\left[\frac{n}{L^2}\right]}.
$$
In order to prove (\ref{Qn}), for
 $j=L^2-L$ and $i=0,1, \ldots, j $ consider the events

$$
A_i=\{\varepsilon: \ (\varepsilon_{i},\ldots, \varepsilon_{i+L-1})=\overline{\varepsilon}^{(L)} \}.
$$
Then, by the inclusion-exclusion principle we have that

\begin{equation}
\label{q}
Q[(D_{0,L^2})^c]\ge \sum_{0\leq j_1\leq j}Q[A_{j_1}]
 -\sum_{0\leq j_1<j_2\leq j }Q[A_{j_1}\cap A_{j_2}].
\end{equation}
Now, note that

\begin{eqnarray}
\nonumber
& \displaystyle{\sum_{0\leq j_1<j_2\leq j }}Q[A_{j_1}\cap A_{j_2}]\leq j\kappa^{L+1}+(j-1)\kappa^{L+2}+\ldots\\
\nonumber
&\ldots +(j-L+1)\kappa^{2L}+(j-L)\kappa^{2L}+\ldots +(j-(j-1))\kappa^{2L}\\
\nonumber
&\leq j \kappa^L  \displaystyle{\sum_{n=1}^L }\kappa^n +\kappa^{2L}(j-L)^2
\leq L^2 \kappa^L  \frac{1-\kappa^{L+1}}{1-\kappa}+L^4\kappa^{2L}\\
\label{qq}
&\leq c_{34}L^2\kappa^L,
\end{eqnarray}
for some constant $c_{34}$. Since $Q[A_{i}]=\kappa^L$ for all $1\le i\le j$,
we conclude from (\ref{q}) and (\ref{qq}) that there is a constant $c_{33}$
such that

\begin{equation*}
Q[D_{0,L^2}]=1-Q[(D_{0,L^2})^c] \leq 1 - c_{33}L^2\kappa^L.
\end{equation*}
This finishes the proof.

\medskip

\section*{Acknowledgement}
The first author wishes to thank Alexander Drewitz for useful conversations
about equivalent formulations of condition $(T')$.


\begin{thebibliography}{99}


\bibitem[BDR14]{BDR14} N. Berger, A. Drewitz and A.F. Ram\'\i rez.
\emph{Effective Polynomial Ballisticity Conditions for Random Walk in Random Environment.} 
 Comm. Pure. Appl. Math. 67, 1947-1973 (2014).

\bibitem[Ch62]{Ch62} A.A. Chernov. \emph{ Replication of a multicomponent chain by the
lightning mechanism}. Biophysics 12, 336-341 (1962).  


\bibitem[CZ01]{CZ01}F. Comets and  O. Zeitouni.
\emph{A law of large numbers for random walks in random mixing environments. }Ann. Probab. 32,  880-914, (2004).

\bibitem[CZ02]{CZ02}F. Comets and  O. Zeitouni.
 \emph{Gaussian fluctuations for random walks in random mixing environments.} Probability in mathematics. Israel J. Math. 148, 87-113,  (2005).


\bibitem[Ch67]{Ch67} K. L. Chung.
\emph{Markov Chains: With Stationary Transition Probabilities.} Grundlehren der mathematischen Wissenschaften, Springer-Verlag Berlin Heidelberg (1967).


\bibitem[DR14]{DR14}  A. Drewitz and A.F. Ram\'\i rez.
\emph{Selected topics
in random walks in random environment}. Topics in percolative and disordered systems, 23–83, Springer Proc. Math. Stat., 69, Springer, New York, (2014).


\bibitem[FGP10]{FGP10} A. Fribergh, N. Gantert and S. Popov.
\emph{On slowdown and speedup of transient random walks in random environment.} Probab. Theory Relat. Fields 147, 43-88, (2010).


\bibitem[K81]{K81} S. Kalikow.
\emph{Generalized random walks in random environment. }Ann. Probab. 9, 753-768, (1981).


\bibitem[KKS75]{KKS75} H. Kesten, M. V. Kozlov and F. Spitzer.
\emph{A limit law for random walk in a random environment.}  Compositio Mathematica 30, 2, 145-168, (1975).

\bibitem[M94]{M94} S. Molchanov. \emph{Lectures on random media}. Lectures on probability theory (Saint-Flour,
1992), 242-411, Lecture Notes in Math., 1581, Springer, Berlin, (1994).

\bibitem[RA03]{RA03} F. Rassoul-Agha.
\emph{The point of view of the particle on the law of large numbers for random walks in a mixing random environment.} Ann. Probab. 31,  1441-1463, (2003).

\bibitem[Si07]{Si07}F. Simenhaus.
\emph{Asymptotic direction for random walks in random environments.} Ann. Inst. H. Poincaré Probab. Statist. 43, 751-761, (2007).


\bibitem[Si82]{Si82}  Y. S. Sinai. \emph{The limiting behavior of a one-dimensional
random walk in a random envrionment}. Theory Prob. Appl. 27, 247-258 (1982).

\bibitem[SZ99]{SZ99} A.S. Sznitman and M. Zerner.
\emph{A law of large numbers for random walks in random environment.}
Ann. Probab. 27,  1851-1869, (1999)

\bibitem[Sz01]{Sz01} A.S. Sznitman. \emph{Slowdown estimates and central limit theorem for random walks in random environment.} J. Eur. Math. Soc. 2,
 93-143 (2000).

\bibitem[Sz02]{Sz02} A.S. Sznitman. \emph{On a class of transient random walks in random environment.} Ann. Probab. 29 (2), 724-765 (2001).

\bibitem[Sz03]{Sz03} A.S. Sznitman. \emph{An effective criterion for ballistic behavior of random walks in random environment.}
 Probab. Theory Relat. Fields 122, 509-544 (2002).

\bibitem[T69]{T69} D.E. Temkin. \emph{The theory of diffusionless crystal growth}.
Journal of Crystal Growth, 5,  193-202 (1969).



\bibitem[W91]{W91} D. Williams. \emph{Probability with martingales.} Cambridge
Univ. Press (1991).


\bibitem[Z04]{Z04} O. Zeitouni.  \emph{Random walks in random environment}. Lectures on probability theory and statistics, 189–312, Lecture Notes in Math., 1837, Springer, Berlin, (2004).

 \end{thebibliography}
\end{document}